%% file: cgr2.tex
\documentclass[a4paper,12pt]{article}

\usepackage{mathtools}
\usepackage[latin1]{inputenc}
\usepackage[all]{xy}
\usepackage{float}
\usepackage{amssymb,amsfonts,amsthm}
\usepackage{enumerate}
\usepackage{epstopdf}
\usepackage{color,fancybox}
\usepackage{url}
\usepackage{graphicx,psfrag,epsfig}
\usepackage{bm}
\include{macros}

\numberwithin{equation}{section}
          
\begin{document}

\begin{center}
  {\sc \Large On Synchronized Fleming-Viot\\ Particle Systems}
  \vspace{0.5cm}
\end{center}

{\bf Fr\'ed\'eric C\'erou\footnote{Corresponding author.}}\\
{\it INRIA Rennes \& IRMAR, France }\\
\textsf{frederic.cerou@inria.fr}
\bigskip

{\bf Arnaud Guyader}\\
{\it LPSM, Sorbonne Universit\'e \& CERMICS, France }\\
\textsf{arnaud.guyader@upmc.fr}
\bigskip

{\bf Mathias Rousset}\\
{\it INRIA Rennes \& IRMAR, France }\\
\textsf{mathias.rousset@inria.fr}
\bigskip

\medskip

\begin{abstract} This article presents a variant of Fleming-Viot particle systems, which are a standard way to approximate the law of a Markov process with killing as well as related quantities. Classical Fleming-Viot particle systems proceed by simulating $N$ trajectories, or particles, according to the dynamics of the underlying process, until one of them is killed. At this killing time, the particle is instantaneously branched on one of the $(N-1)$ other ones, and so on until a fixed and finite final time $T$. In our variant, we propose to wait until $K$ particles are killed and then rebranch them independently on the  $(N-K)$ alive ones. Specifically, we focus our attention on the large population limit and the regime where $K/N$ has a given limit when $N$ goes to infinity. In this context, we establish consistency and asymptotic normality results. The variant we propose is motivated by applications in rare event estimation problems.\medskip

  \medskip


\end{abstract}

\newpage
\tableofcontents

\section{Introduction}\label{intro}

Let $X=(X_t)_{t\geq 0}$ denote a Markov process evolving in a state space of the form $ F \cup \set{\partial}$, where $\partial \notin F$ is an absorbing state: $X$ evolves in $F$ until it reaches $\partial$ and then remains trapped there forever. $X$ is called a killed Markov process with cemetery point $\partial$. Let us also denote $\tau_\partial$ the associated killing time, meaning that
$$\tau_{\partial}\eqdef  \inf\{t\geq 0, X_t=\partial \}.$$
Given a deterministic final time $T>0$, we are interested both in the distribution of $X_T$ given that it is alive at time $T$, i.e., ${\cal L}(X_T|\tau_{\partial }>T)$, and in the probability of this event, that is $p_T\eqdef  \P(\tau_{\partial }>T)$, with the natural assumption that $p_T>0$ (see Figure \ref{fv1}). Without loss of generality, we will assume for simplicity that $\P(X_0 = \partial)=0$, that is $p_0 = 1$. Let us stress that in all this paper, $T$ is held fixed and finite.\medskip 

\begin{figure}
\begin{center}
\input{fv1.pstex_t}
\caption{A Markov process with killing.}
\label{fv1}
\end{center}
\end{figure}
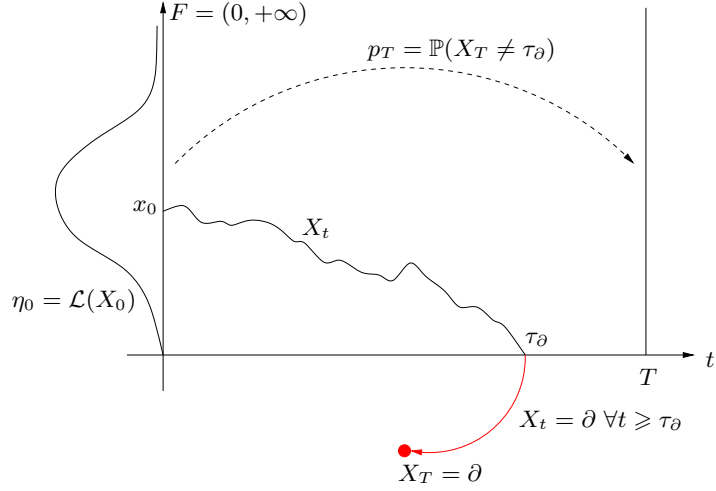

%{\tiny 
%
%A crude Monte Carlo method approximating these quantities consists in:
%\begin{itemize}
%\item simulating $N$ i.i.d.~~random variables, also called particles in the present work, 
%$$X_0^1,\dots,X_0^N\ \overset{\rm i.i.d.~}{\sim}\ \eta_0,$$ 
%\item letting them evolve independently according to the dynamic of the underlying process $X$, 
%\item and eventually considering the estimators
%$$\hat{\eta}_T^N\eqdef  \frac{\sum_{i=1}^N \un_{X_T^i\in F}\ \delta_{X_T^i}}{\sum_{i=1}^N\un_{X_T^i\in F}}\hspace{1cm}\mbox{and}\hspace{1cm}\hat{p}_T^N\eqdef  \frac{\sum_{i=1}^N\un_{X_T^i\in F}}{N},$$
%with the convention that $0/0= 0$.
%\end{itemize}
%It is readily seen that these estimators are not relevant for large $T$, typically when $T\gg\E[\tau_{\partial }]$, since one has then to face a rare event estimation problem. A possible way to tackle this issue is to approximate the quantities at stake through a Fleming-Viot type particle system \cite{bhim96,v14}.
%
%}

We will also assume -- as a consequence of Assumption \ref{ass:A} below -- that the non-increasing function $t \mapsto p_t$ is continuous on $[0,T]$, and we will consider the approximation $\rho_t$ of $p_t$ defined by
$$
t \mapsto \rho_t \eqdef \theta^{ \lfloor \log p_t / \log \theta \rfloor },
$$
where $\theta \in (0,1)$ is a given probability typically much larger than $p_T$. To fix ideas, one can think of $p_T$ being lower than $10^{-6}$ while $\theta=1/2$. In other words, $t \mapsto \rho_t$ is the right continuous and piecewise constant function that coincides with $p_t$ for each power of $\theta$, that is (see Figure \ref{pt})
$$
\rho_t = \theta^j\ \Longleftrightarrow\  \theta^{j+1} < p_t \leq \theta^j, \, j \in \N.
$$
Let us denote $$\jmax=\lfloor \log p_T / \log \theta \rfloor.$$ For simplifying technical reasons, we assume that $\log p_T / \log \theta$ is not an integer, so that $p_T=r\theta^{\jmax}$ with $\theta<r<1$. Moreover, we suppose that for each $0\leq j\leq \jmax$, there exists a unique $t_j$ such that 
$$p_{t_j}=\P(\tau_\partial > t_j )=\theta^j.$$ 
This is obviously true if the non-increasing function $t \mapsto p_t$ is in fact strictly decreasing on $[0,T]$ (see Figure \ref{pt}).\medskip

\begin{figure}
\begin{center}
\input{pt.pstex_t}
\caption{The mappings $t\mapsto p_t$ and $t\mapsto \rho_t$ when $\theta=1/2$.}
\label{pt}
\end{center}
\end{figure}
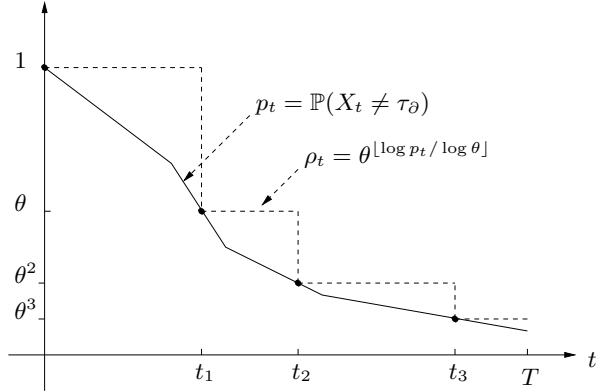

Under Assumptions~\ref{ass:A} and~\ref{ass:B} that will be detailed below, the following process is well defined for any number of particles $N\geq 2$ and any $1\leq K<N$. We propose to call it a Synchronized Fleming-Viot Particle System (see Figure \ref{fv4}). The ``classical'' Fleming-Viot Particle System corresponds to the case where $K=1$, see for example \cite{bhim96,GK04,lobus,GK12,BBF12,v14,cdgr1,cdgr2} and references therein.

\begin{Def}[Synchronized Fleming-Viot Particle System]\label{def:ips} Let $X$ denote a killed Markov process in $F$ with cemetery point $\partial$. The associated synchronized Fleming-Viot particle system $(X^1_t, \cdots, X^N_t)_{t \in [0,T]}$ with $\k$ synchronized branchings is the Markov process with state space $\p{ F\times \set {\partial} }^N$ defined by the following set of rules:
\begin{itemize}
\item Initialization: consider $N$ i.i.d.~~particles
\begin{equation*}
X_0^1,\dots,X_0^N\ \overset{\rm i.i.d.~}{\sim}\ \eta_0={\cal L}(X_0),
\end{equation*}
\item Evolution and killing: each particle evolves independently according to the law of the underlying Markov process $X$ until $\k$ of them hit  $\partial$ (or the final time $T$ is reached),
\item Branching (or rebirth, or splitting): the $K$ killed particles are taken from $\partial $, and are independently and instantaneously given the state of one of the $(N-\k)$ other particles (randomly uniformly chosen),
\item and so on until final time $T$.
\end{itemize}
\end{Def}

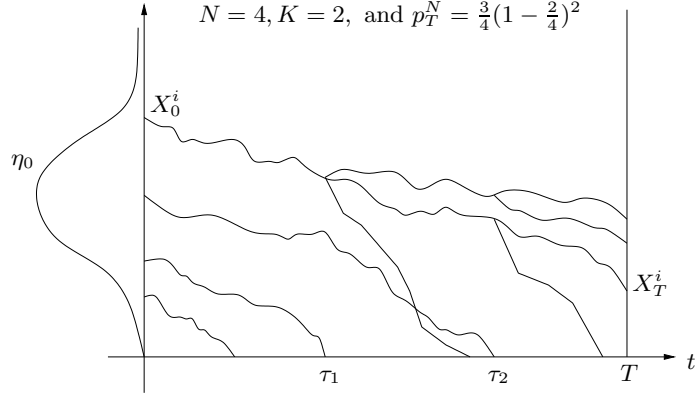
\begin{figure}
\begin{center}
\input{fv4.pstex_t}
\caption{A Synchronized Fleming-Viot Particle System.}
\label{fv4}
\end{center}
\end{figure}

As will be proved in Proposition~\ref{pro:quant}, it turns out that the sequence of quantiles $(t_j)_{0< j \leq \jmax}$ are approximated by the sequence of successive branchings times $(\tau_j)_{0<j \leq \jmax}$ of the synchronized Fleming-Viot particle system when $\k := \k_N$ satisfies
\begin{equation}\label{eq:scale}
1-\frac{\k_N}{N} \xrightarrow[N\to\infty]{} \theta.
\end{equation}

Additionally, we define the right continuous counting process
$$
t \mapsto \calB_t =\calB_t^N\eqdef \Card \set{\text{branching times}\leq t}.
$$
and then consider the estimators
$$\eta_t^N\eqdef  \frac{1}{N}\sum_{n=1}^N\delta_{X_t^n}\hspace{1cm}\mbox{and}\hspace{1cm} \rho_t^N\eqdef  \left(1-\tfrac{\k}{N}\right)^{\calB_t}, $$
of 
$$\eta_t \eqdef {\cal L} \p{X_t | \tau_\partial > t_j} \quad \forall\ t_j \leq t < t_{j+1},$$
  and $\rho_t$ respectively. Since we have assumed that $p_0 = 1$, we have $p_0=\rho_0$ and $\eta_0 = {\cal L}(X_0)$. Note that the branching times $\tau_j$ estimating the quantiles $t_j$ are implicitly estimated by the jump times of the process $ t \mapsto \rho^N_t$. Moreover, we emphasize that $\eta_t$ is not ${\cal L}(X_t|\tau_\partial>t)$, which is the law of the process $X_t$ given that it is still alive at time $t$. In particular, $\eta_t$ does not define a probability measure on $F$, but on $F\cup\{\partial\}$ with a Dirac at $\partial$ associated with the probability $\P(\tau_\partial\leq t|\tau_\partial>t_j)$ for all $t$ such that $t_j \leq t < t_{j+1}$.\medskip

The distribution of the process restricted to $F$, that is
$$
\gamma_t(\ph) \eqdef \E[ \ph(X_t)] \quad \forall \ph \text{\, with \,} \ph(\partial) =0,
$$
is then estimated by
\begin{equation}\label{ljsnclnjc}
\gamma^N_t \eqdef \rho_t^N \times \eta_t^N.
\end{equation}
The probability $p_t$ that the process is still alive at time $t$ is estimated by (see also Figure \ref{fv4})
$$
p^N_t \eqdef \gamma^N_t( \un_F) = \rho_t^N \times \eta_t^N( \un_F)=\left(1-\frac{\k}{N}\right)^{\calB_t}\frac{1}{N}\sum_{n=1}^N\un_F(X_t^n),
$$
and the distribution of $X_t$ conditioned to be alive, i.e., ${\cal L}(X_t|\tau_\partial>t)$, by
$$
\frac{1}{\gamma^N_t(\un_F)}\gamma^N_t = \frac{1}{\eta^N_t(\un_F)}\eta^N_t .
$$
Note that this notation differs from the work in~\cite{cdgr2} and of classical particle models with fixed population size where the empirical distribution of particles $\eta_t^N$ estimates the  
law of $X_t$ conditioned to be alive at time $t$. In particular, in the present work, the mapping $t \mapsto \eta_t(\un_F)$ is \cadlag with jumps at the quantiles $t_j$ with $\Delta \eta_{t_j}(\un_F) =1- \theta$.\medskip
% 
% In the same spirit, we will define the non-normalized measure $\gamma^N_t = \rho^N_t \eta^N_t$ so that $ \gamma^N_t(\un) = \rho_t^N $ is piecewise constant between the quantiles, while $\gamma^N_t(\un_F) = p^N_t$ estimates the probability to be alive at time $t$. \medskip

The purpose of this paper is to extend the results obtained in~\cite{cdgr2} for classical Fleming-Viot particle systems to the synchronized Fleming-Viot particle systems that we have just defined under the scaling assumption \eqref{eq:scale}, that is when $\k=\k_N$ is proportional to $N$. The main result corresponds to Theorem~\ref{gamma} and Corollary~\ref{eta}, which provide CLT type results for the estimators $\gamma^N_T$, $p^N_T$ and $\eta^N_T / \eta^N_T(\un_F)$. We also prove convergence of the branching times $\tau_j$ towards the corresponding quantiles $t_j$ in Proposition \ref{pro:quant}. \medskip

We refer to~\cite{cdgr2} for examples where Assumptions~\ref{ass:A} and~\ref{ass:B} below are verified. In particular, this includes the case where $X_t$ is a regular enough uniformly elliptic diffusive process killed when hitting the smooth boundary of a given compact domain. \medskip

Finally, note that this work is motivated by practical applications in rare event estimation problems. We refer to \cite{Au2001263,au:901,cg2,cglp,ghm,cdfg} for the presentation of the set of methods we have in mind, called Adaptive Multilevel Splitting or Subset Simulation. More precisely, we prove in~\cite{cdgr3} that splitting algorithms can be interpreted as Fleming-Viot particle systems, using a ``score function'' (also called an importance function or a reaction coordinate) as a new time-index. The present paper enables us to obtain CLT type results for versions of this algorithm where the $K$ particles with minimal score are killed at each step. The interested reader can find applications of Adaptive Multilevel Splitting in various fields in \cite{cfgjournal,teo2016adaptive,Lestang,MR3924366,MR3937660,tutorialAMS} and references therein.\medskip

The rest of the paper is organized as follows. Section~\ref{AMEKCN} details our assumptions and exposes the main results of the paper, Section \ref{mazlco} is dedicated to the proofs while Section \ref{akjajdnzjd} gathers some supplementary material.

%
%\medskip
%
%The rest of the paper is organized as follows. Section 2 details our assumptions, exposes the main results of the paper, and illustrates a possible context of application for a process with hard killing. Section 3 is dedicated to the proof of the central limit theorem, while Section 4 gathers some technical results. 
%{\tiny
%
%where $N \calN_T / \k$ is the total number of branchings of the particle system until final time $T$. In other words, $\calN_T$ is the empirical mean number of branchings per particle until final time $T$:
%$$
%\calN_T \eqdef \frac{\k}{N} \Card \set{\text{branching times}\leq T}.
%$$
%
%
%Under very general assumptions, Villemonais~\cite{v14} proves among other things that $p_T^N$ (or equivalently $e^{- \calN_T }$) converges in probability to $p_T$ when $N$ goes to infinity, and that $\eta_T^N$ converges in law to $\eta_T$. In~\cite{cdgr1}, we went one step further and established central limit results for $\eta_T^N$ and $p_T^N$. For this, we had to make two specific assumptions. The first one is a ``soft killing'' assumption, meaning that one can define a bounded intensity of being killed when the process is at point $x\in F$. The second one is a so-called ``bounded carr\'e du champ'' assumption and is related to the regularity of the underlying Markov process.\medskip
%
%
%
%}
%
%
\section{Main result}\label{AMEKCN}
\subsection{Assumptions}\label{zcijo}

We assume that $F$ is a measurable subset of some reference Polish space, and that for each initial condition, under this reference topology, $X$ is \cadlag in $F \cup \set{\partial}$ and satisfies the time-homogeneous Markov property. Its probability transition is denoted $Q$, meaning that there is a semi-group operator $(Q^t)_{t \geq 0}$ defined for any bounded measurable function $\ph:F\to\R$, any $x\in F$ and any $t\geq 0$, by
$$Q^t\ph(x) \eqdef \E[\ph(X_t)|X_0=x].$$
By convention, in the latter and in all what follows, the test function $\ph$ defined on $F$ is extended on $F\cup\{\partial\}$ by setting $\ph(\partial)=0$. Thus, we have $Q^t\ph( \partial )=0$  for all $t\geq 0$. This equivalently defines a sub-Markovian semi-group on $F$ that is also denoted $(Q^t)_{t \geq 0}$. \medskip

For any bounded $\ph:F \to\R$ extended with the convention $\ph(\partial)=0$ and any $t \in [0,T]$, we can then consider the unnormalized measure
$$\gamma_t(\ph) = \E[\ph(X_t)] = \E[\ph(X_t)\un_{\tau_\partial>t}]= \eta_0 Q^{t} \ph ,$$
with $X_0\sim\eta_0=\gamma_0$. For any $t\in[0,T]$, one has $p_t=\P(\tau_\partial >t)=\gamma_t(\un_F)$. As mentioned before, the associated empirical approximation is given by
$$\gamma_t^N\eqdef  \rho_t^N \eta_t^N,$$
so that we can define the limiting measure
$$
\eta_t \eqdef  \gamma_t /\rho_t.
$$
We stress again that, contrary to~\cite{cdgr2}, $\eta_{t}(\un_F) \neq 1$ in general but since we have assumed that the process $X_t$ is \cadlag we still have  $ \eta_{t_j}(\un_F)=1$ and
$$
 \eta_{t_j} = {\cal L} \p{ X_{t_j} | \tau_\partial > t_j }.
$$

We recall that, when $X$ is a time-homogeneous Markov process, the process $t \mapsto Q^{T-t}(\ph)(X_t)=\gamma_t(Q^{T-t}(\ph))$ is a \cadlag martingale on $[0,T]$ with respect to the natural filtration of $X$. Our fundamental assumptions can now be detailed.

\begin{AssA}\label{ass:A} Let $X$ denote a time-homogeneous Markov process.
\begin{enumerate}
\item[(i)] For any initial condition $X_0=x \in F$, the distribution of the killing time $\tau_\partial$ is atomless.
\item[(ii)] There exists a space $\calD$ of bounded measurable real-valued functions on $F$,
which contains at least the indicator function $\un_F$, and such that for any $\ph \in \calD$, any initial condition $X_0=x$ and any final time $T$, the jumps of the  \cadlag martingale $t \mapsto Q^{T-t}(\ph)(X_t)$ have an atomless distribution.\end{enumerate}
\end{AssA}

%{ \tiny 
%
%
%\begin{Rem}\label{alzichachi}
%Note that Conditions (i) and (ii) in Assumption~\ref{ass:A} both imply that for any initial condition $x \in F$, the killing time $\tau_\partial$ has an atomless distribution in $[0,+\infty)$. Indeed, for (i), if $t=\tau_\partial$ then obviously $X_{t^-} \neq X_t$ and we conclude that this event happens with probability 0 at any deterministic time $t$. Equivalently, taking $\ph = \un_F$ in~$(ii)$ implies that $t \mapsto \P\p{ \tau_\partial > t |X_0=x}$ is continuous. Note that $ \tau_\partial = + \infty$ may have positive probability.
%\end{Rem}
%
%\begin{Rem}
%In Section \ref{sec:wellposed}, we present a weaker but less practical version of Assumption~\ref{ass:A}, named Assumption~\ref{ass:Ap}. Lemma \ref{mzoecj} ensures that \ref{ass:A} implies \ref{ass:Ap}. As will be explained, all the results of the present paper are in fact obtained under Assumption~\ref{ass:Ap}.
%\end{Rem}
%
%}

%
%
%
%
%
%
%
%

Our second assumption ensures the  existence of the particle system at all time.
\begin{AssB}\label{ass:B} The particle system of Definition~\ref{def:ips} is well-defined in the sense that $\P(\calB_T < + \infty)=1$.
\end{AssB}

Under Assumptions~\ref{ass:A} and~\ref{ass:B}, the non-increasing jump processes  $t \mapsto p_t^N$ and $t \mapsto \rho_t^N$ are strictly positive.

\begin{Rem}\label{lem:pcont}
Condition~$(i)$ of Assumption~\ref{ass:A} and Lebesgue's continuity theorem imply that, for any initial distribution $\eta_0$ on $F$, the non-increasing mapping $t \mapsto p_t=\P\p{\tau_\partial > t}$ is continuous. \medskip
\end{Rem}

Our third and final assumption ensures the strict monoticity of $t \mapsto p_t$ at each quantile.

\begin{AssC}\label{ass:C} The continuous non-increasing mapping $t \mapsto p_t$ is strictly decreasing at each $t_j$, $j =1 \ldots \jmax$.
\end{AssC}

%\mathias{D'après une remarque de Fred: il faut ajouter ici la preuve que $t \mapsto \gamma_t(Q^2)$ est continue, un point qui manquait dans cdgr2. N.B.: Ca se démontre avec CVD et hypothèse A  qui demande à ce que la martingale $Q^{T-t}(\ph)(X_t)$ ait ses sauts de loi sans atome.}
%
%
%{\tiny 
%\begin{proof}
%As mentioned in Remark \ref{alzichachi}, Assumptions~\ref{ass:A}$(i)$ and \ref{ass:A}$(ii)$ both ensure the continuity of 
%$t \mapsto \P\p{\tau_\partial > t|X_0=x} $ for all $x\in F$.  
%And the continuity result now comes from $p_t=\P(\tau_\partial >t)=\int \P(\tau_\partial>t| X_0=x) \eta_0(dx)$. The proof under Assumption~\ref{ass:Ap}(i) is similar. Besides, recall that $p_T$ is strictly positive by assumption. The subsequent assertions are clearly satisfied by definition of $p_t^N$.
%%
%%
%%
%\end{proof}
%
%
%\begin{Rem}
% We will see in Lemma~\ref{lem:decomp} that under Assumptions~\ref{ass:A} and~\ref{ass:B}, one has $\E[p^N_T] = p_T$, which implies in particular that those assumptions cannot be true in the degenerate case where $p_T = 0$.
%\end{Rem}
%
%}

\subsection{Main result}\label{sec:main}

We keep the notation of Section~\ref{intro}. In particular, $(X^1_t, \ldots, X^N_t)_{t\geq 0}$ denotes the synchronized Fleming-Viot particle system, and
$$\tau_{j} \eqdef \, \text{$j$-th branching time of the particle system}.$$
Accordingly, the number ${\calB_t}$ of branchings until time $t$  is ${\calB_t}=\sum_{j=1}^\infty\un_{\tau_j\leq t}$.\medskip

%Idem for Section~\ref{zcijo}, in particular $Q^t$ is the probability transition restricted to $F$ of $X$ and $\eta_{t_j}$ is the distribution at quantiles conditioned by survival. \medskip
%{\tiny
%
% Accordingly, the processes
%$
%\calN^{n}_t\eqdef   \sum_{k\geq 1} \un_{\tau_{n,k}\leq t}
%$
%and  
%$$
%\calN_t\eqdef   \frac1N \sum_{n=1}^N \calN^{n}_t = \frac{\k}{N} \sum_{j\geq 1} \un_{\tau_{j}\leq t}
%$$
%are \cadlag counting processes that correspond respectively to the number of branchings of particle $n$ before time $t$, and to the total number of branchings per particle of the whole particle system before time $t$.\medskip
%
%
%As mentioned before, we can then define the empirical measure associated to the particle system as 
%$ \eta^N_t\eqdef   \frac{1}{N} \sum_{n=1}^N \delta_{X_t^n},$
%while the estimate of the probability that the process is still not killed at time $t$ is denoted 
%$p^N_t\eqdef   (1-\tfrac{\k}{N})^{N \calN_t},$
%and the unnormalized empirical measure is defined as $\gamma^N_t\eqdef   p^{N}_t \eta^N_t$.\medskip
%
%As will be recalled in Proposition \ref{pro:estimate} and already noticed by Villemonais in \cite{v14}, their large $N$ limits are respectively 
%%
%$$\eta_t(\varphi)\eqdef\E [\varphi(X_t)|X_t\neq\partial],\;
%p_t\eqdef \P(X_t\neq\partial), \mbox{ and }
%%
%\gamma_t(\varphi) \eqdef \E[\varphi(X_t)\un_{X_t\neq\partial}].$$
%%
%We clearly have $\eta_t(\varphi)=\gamma_t(\varphi)/\gamma_t(\un_F)=\gamma_t(\varphi)/p_t$ and $\gamma_t(\varphi)=\eta_0(Q^t \varphi)$.\medskip 
%}

 For the upcoming results, we work under Assumptions~\ref{ass:A}, \ref{ass:B}, and~\ref{ass:C}, and we assume that $\k = \k_N$ satisfies
 $$
  \frac{\k_N}{N} \xrightarrow[N\to\infty]{} 1-\theta\in(0,1).
 $$
We start with the convergence of the branching times towards the quantiles.

\begin{Pro}
\label{pro:quant}
We have
$$(\tau_1, \ldots ,\tau_{j_{\rm max}}) \xrightarrow[N \to + \infty]{\P} (t_1, \ldots, t_{j_{\rm max}})\hspace{1cm}\mathrm{and}\hspace{1cm}\P(\tau_{\jmax+1}\leq T)\xrightarrow[N \to + \infty]{} 0.$$
\end{Pro}

We can now expose the main result of the present paper. As usual, $\calN(m,\sigma^2)$ denotes the normal distribution with mean $m$ and variance $\sigma^2$. Furthermore, for any probability distribution $\mu$ on $F$ and any test function $\ph:F\to\R$, the standard notation $\Var_\mu(\ph)$ stands for the variance of the random variable $\ph(Y)$ when $Y$ is distributed according to $\mu$, i.e.,
$$\Var_\mu(\ph)\eqdef  \Var(\ph(Y))=\E[\ph(Y)^2]-\E[\ph(Y)]^2=\mu(\ph^2)-\mu(\ph)^2.$$

\begin{The}\label{gamma}
Let us denote by $\overline{\cal D}$ the closure with respect to the norm $\norm{\cdot}_{\infty}$ of the space ${\cal D}$ satisfying Condition~$(ii)$ of Assumption~\ref{ass:A}. Then for any $\ph$ in $\overline{\cal D}$ extended with $\ph(\partial)=0$, one has the convergence in distribution
$$\sqrt{N}\left(\gamma_T^N(\ph)-\gamma_T(\ph)\right)\xrightarrow[N\to\infty]{d}{\cal N}(0,\sigma_T^2(\ph)),$$
where $\sigma_T^2(\ph)$ is defined by
\begin{align}
\sigma_T^2(\ph)  &  \eqdef  \quad \theta^{2 \jmax} \V_{\eta_T}(\ph)+  \jmax(1/\theta -1) \theta^{2 \jmax}\eta_T\p{\ph}^2  \nonumber \\
 & \qquad +  \sum_{j=1}^{\jmax} \V_{\eta_{t_j}}\p{Q^{T-t_j}(\ph)} \nonumber \p{\theta^{2j-1}-\theta^{2j+1}}. \label{eq:variance} 
\end{align}

%$$\sigma_T^2(\ph) \eqdef  \Var_{\eta_0}(Q) + \sum_{j=0}^\jmax \theta^{j}(\gamma_{t_{j+1}}(Q^2)-\gamma_{t_j}(Q^2))+\sum_{j=1}^\jmax \theta^{2j}(1-\theta) \V_{\eta_{t_j}}(Q)$$

%{\tiny 
%OLD FORMULA: 
%$$ \sigma_T^2(\ph)  =  p^2_T \Var_{\eta_T}(\ph) - p_T^2\log(p_T) \, \eta_T(\ph)^2 - 2\int_0^T \Var_{\eta_{t}}(Q^{T-t}(\ph)) p_t dp_t.$$}
\end{The}

For classical Fleming-Viot particle systems, we have shown in \cite{cdgr2} that, under the same assumptions and denoting $\eta_t={\cal L}(X_t|t>\tau_\partial)$, the asymptotic variance takes the form
\begin{equation}\label{kabcxkbx}
\sigma_T^2(\ph)=p_T^2\V_{\eta_T}(\ph)-p_T^2\log(p_T)\eta_T(\ph)^2-2\int_0^T  \V_{\eta_{t}}\p{Q^{T-t}(\ph)}p_tdp_t.
\end{equation}

Returning to our variant and the result of Theorem~\ref{gamma}, since $\un_F\in{\cal D}$ by assumption, and $\gamma_T(\un_F)=p_T$, the CLT for $\eta_T^N / \eta_T^N(\un_F)$ is then a straightforward application of this result by considering the decomposition 
\begin{align*}
& \sqrt{N}\left( \frac{\eta_T^N\p{\ph}}{\eta_T^N(\un_F)}-\frac{\eta_T(\ph)}{\eta_T(\un_F)}\right)= \\
& \quad \frac{1}{\gamma_T^N(\un_F)} \sqrt{N}\left(\gamma_T^N(\ph-\eta_T\left(\ph\right)\un_F / \eta_T( \un_F) )-\gamma_T(\ph-\eta_T\left(\ph\right)\un_F / \eta_T( \un_F))\right),
\end{align*}
and the fact that $\gamma_T^N(\un_F)$ converges in probability towards $p_T=\gamma_T(\un_F)$.

\begin{Cor}\label{eta}
One has 
$$\sqrt{N}\left(p_T^N-p_T\right)\xrightarrow[N\to\infty]{d}{\cal N}(0,\sigma_T^2(\un_F)).$$

Additionally, for any $\ph$ in $\overline{\cal D}$, 
$$\sqrt{N}\left( \frac{\eta_T^N\p{\ph}}{\eta_T^N(\un_F)}-\frac{\eta_T(\ph)}{\eta_T(\un_F)}\right)  \xrightarrow[N\to\infty]{d}{\cal N}\p{0,\sigma_T^2(\ph-\eta_T\left(\ph\right)\un_F / \eta_T( \un_F)) / p_T^{2}}.$$
\end{Cor}

Let us comment on the relative asymptotic variance of $p_T^N$, i.e., $\sigma_T^2(\un_F)/p_T^2$. Since by definition $p_T=r \theta^{\jmax}$, with $\jmax=\lfloor \log p_T / \log \theta \rfloor$ and $\theta<r<1$, we have $\eta_T(\un_F)=r$ and $\V_{\eta_T}(\un_F)=r(1-r)$.  Therefore, we obtain
$$\frac{\sigma_T^2(\un_F)}{p_T^2}=\jmax\frac{1-\theta}{\theta}+\frac{1-r}{r}+\frac{1}{p_T^2} \sum_{j=1}^{\jmax} \V_{\eta_{t_j}}\p{Q^{T-t_j}(\un_F)} \nonumber \p{\theta^{2j-1}-\theta^{2j+1}}. $$
In the latter, the variance terms may be reformulated as
$$\V_{\eta_{t_j}}\p{Q^{T-t_j}(\un_F)}=\V(\P(X_T\neq\partial|X_{t_j}))=\E\left[\left(\P(X_T\neq\partial|X_{t_j})-\frac{p_T}{p_{t_j}}\right)^2\right].$$
For each $1\leq j\leq \jmax$,  if $X_{t_j}\sim\eta_{t_j}$, $\P(X_T\neq\partial|X_{t_j})$ is a random variable with values between 0 and 1 and expectation $p_T/p_{t_j}$ so that the maximal variance is reached by a Bernoulli random variable with parameter $p_T/p_{t_j}$. As a consequence,
$$0\leq \V_{\eta_{t_j}}\p{Q^{T-t_j}(\un_F)}\leq \frac{p_T}{p_{t_j}}\left(1-\frac{p_T}{p_{t_j}}\right)=r\theta^{\jmax-j}\left(1-r\theta^{\jmax-j}\right).$$
Then a straightforward computation yields 
$$\jmax\frac{1-\theta}{\theta}+\frac{1-r}{r}\leq\frac{\sigma_T^2(\un_F)}{p_T^2}\leq\frac{1+\theta}{p_T}-\frac{\theta+r}{r}-\jmax(1-\theta).$$
On the one side, concerning the upper-bound, we see that
$$\lim_{\theta\to 1^-}\frac{1+\theta}{p_T}-\frac{\theta+r}{r}-\jmax(1-\theta)=2\frac{1-p_T}{p_T}+\log p_T.$$
Interestingly, considering (\ref{kabcxkbx}), this limit corresponds to the upper-bound obtained for classical Fleming-Viot particle systems, that is when $K=1$, as shown in \cite{cdgr2}. In particular, if $p_T$ is very low, this is approximately equal to $2(1-p_T)/p_T$. Clearly, this is twice the relative variance of a naive (or standard) Monte Carlo simulation%, where one would just simulate $N$ i.i.d. initial conditions, wait until final time $T$ and count the number of particles that are still alive. Clearly, this random number has a binomial distribution with parameters $N$ and $p_T$, hence the relative variance $(1-p_T)/p_T$. 
.  Consequently, one has to pay attention to the fact that, in some very specific situations, Fleming-Viot particle systems may lead to very poor estimators.\medskip

On the other side, the lower-bound already appeared in the context of Adaptive Multilevel Splitting (see for example \cite{cg2,cdfg,cg4}). This bound is reached when, for each $j$, starting with law $\eta_{t_j}$ at time $t_j$, the probability of being still alive at time $T$ is constant on the support of $\eta_{t_j}$. Everything happens as if one would estimate independently $\jmax$ times the probability $\theta$ and one time the probability $r$, all this being done by naive Monte Carlo with independent samples of common size $N$. Using standard tools, one can see that the resulting product estimator $\hat p_T^N=\hat p_{t_1}^N\dots\hat p_{t_\jmax}^N\hat r^N$ satisfies the CLT
$$\sqrt{N}\ \frac{\hat p_T^N-p_T}{p_T}\xrightarrow[N\to\infty]{d}{\cal N}\left(0,\jmax\frac{1-\theta}{\theta}+\frac{1-r}{r}\right),$$
which is exactly the above-mentioned lower-bound. %In other words, the relative asymptotic variance of the product is the sum of all relative variances when the estimators are independent.
\medskip

Additionally, since  $\jmax=\lfloor \log p_T / \log \theta \rfloor$ and $\theta<r=p_T\theta^{-\jmax}<1$, this lower-bound may be rewritten, for all $0< \theta<1$, as
$$h(\theta):=\left\lfloor \frac{\log p_T }{\log \theta} \right\rfloor \frac{1-\theta}{\theta}+\frac{\theta^{\left\lfloor \frac{\log p_T }{\log \theta} \right\rfloor}}{p_T}-1.$$
One can check that $h$ is non increasing on $(0,1)$, so that the minimal possible relative asymptotic  variance is simply
$\lim_{\theta\to 1^-}h(\theta)=-\log p_T$. As explained in \cite{cdgr2} and can be deduced from (\ref{kabcxkbx}), this precisely corresponds to the minimal relative asymptotic variance for classical Fleming-Viot particle systems, that is when $K=1$.\medskip

In view of this, one could argue that the best thing to do for estimating $p_T$ is to use classical Fleming-Viot particle systems. However, things are not so simple. First, because as far as we can judge, this is only the case when the variance terms $\V_{\eta_{t}}\p{Q^{T-t}(\un_F)}$ are zero at all time, which is a very particular situation. Second, because as for all Monte Carlo methods, one should not only compare the variances of different methods, but also their respective algorithmic complexities, or costs.\medskip

More explicitly, suppose that, by convention, the algorithmic cost for the simulation of a trajectory/particle until its killing is equal to 1. For classical Fleming-Viot particle systems, we have proved in \cite{cdgr2} that the number of resampling is in $O_p(-N\log p_T)$, so that the total cost is in $O_p(N(1-\log p_T))$, i.e., $N$ initial trajectories plus $O_p(-N\log p_T)$ rebranched ones. For synchronized Fleming-Viot particle systems, Remark \ref{rem:niter} ensures that the number of resamplings goes to $\jmax=\lfloor \log p_T / \log \theta \rfloor$ in probability. Since $K\sim(1-\theta)N$ particles are rebranched at each step, the total cost is asymptotically equivalent to 
$$N\left(1+\left\lfloor \frac{\log p_T }{\log \theta} \right\rfloor (1-\theta)\right),$$
which is less than $N(1-\log p_T)$ for any $0< \theta<1$. Beyond this lower algorithmic cost, it is also worth noting that synchronized Fleming-Viot particle systems can easily be parallelized, contrary to classical ones.
\section{Proof}\label{mazlco}
\subsection{Overview}
The probability space is filtered by the natural filtration of the particle system, denoted $\p{\calF_t}_{t\geq 0}$. Note that $\calF_t$ contains all the events related not only to the trajectories of the particles, but also all the auxiliary variables used for the resamplings, up to time $t$. \medskip

The key object of the proof is the c\`adl\`ag martingale 
\begin{equation*}
t \mapsto \gamma^N_t \p{Q} \eqdef \gamma^N_t \p{Q^{T-t}(\ph)},
\end{equation*}
the fixed parameters $T$ and $\ph$ being implicit in order to lighten the notation. Note that, since $\gamma^N_0=\eta^N_0$ and $\gamma_0=\eta_0$,
\begin{align*}
\gamma_T^N(\ph)-\gamma_T(\ph)=\Big(\gamma_T^N(Q)-\gamma^N_0(Q)\Big)
+\Big(\eta^N_0(Q^T(\ph))-\eta_0(Q^T(\ph))\Big)
\end{align*}
is the final value of the latter martingale, with the addition of a second term depending on the initial condition. This second term satisfies a CLT by assumption. We will handle the distribution of $\gamma_T^N(Q)$ in the limit $N\to\infty$ by using 
a Central Limit Theorem for continuous time martingales, namely Proposition \ref{pro:clt}. 
However, this requires several intermediate steps, mainly for the calculation of the quadratic variation $N[\gamma^N\p{Q},\gamma^N\p{Q}]_t$. In the sequel, we will make extensive use of stochastic calculus for \cadlag semimartingales, as presented for example in \cite{protter} chapter II,  or \cite{js03}.\medskip

%

%
%
%
%

 %
%
%
%

%
%
%
%
%
%

%
%
%
%
%
%
%
%
%
%
%{\tiny
%
%Unfortunately, showing the convergence of this quadratic variation is not easy. Specifically, it is much more difficult than in \cite{cdgr1} where, thanks to the so-called ``carré-du-champ'' and ``soft killing'' assumptions, we could write the predictable quadratic variation as an integral against Lebesgue's measure in time, with bounded integrand. We could then easily show the pointwise convergence of the integrand and apply dominated convergence. 
%Here we cannot do that. Instead, the key idea is to replace the quadratic variation by an adapted increasing process $i_t^N$ such that $N[\gamma^N\p{Q},\gamma^N\p{Q}]_t-i_t^N$ is a local martingale. Finally, the convergence of $i_t^N$ requires some appropriate timewise integrations by parts formulas, as well as the uniform convergence in time of $p_t^N$ to $p_t$. 
%\medskip
%}
%
%
%

We adopt the standard notation $\Delta X_t=X_t-X_{t^-}$ and, to shorten  the notation, we will denote for $l= 1,2$,
\begin{align}\label{eq:M_t}
 \gamma^N_t \p{Q^l} \eqdef \gamma^N_t \p{ \b{Q^{T-t}(\ph)}^l}.
\end{align}
We will also denote, for each $1 \leq n \leq N$ and any $t\in[0,T]$,
\begin{align}
&\L^n_t \eqdef Q^{T-t}(\ph)(X_{t}^n),%
&\L_t \eqdef \frac{1}{N}\sum_{n=1}^N\L^n_t = \eta^N_t(Q),
\label{qlsjnclscn}
\end{align}
where, again, the fixed parameters $T$ and $\ph$ are implicit.

\subsection{Martingale decomposition}\label{sec:martdec}
Let us recall that $\tau_{j}$ denotes the $j$-th branching time of the particle system. We will need some additional notation related to the behavior of the particle system at each branching time.
\begin{Def}\label{def:tau} $\,$
\begin{itemize}
 \item(Individual indexation of branching times) For any $n \in \set{1, \ldots , N}$ and any $k \geq 0$, we denote by $$\tau_{n,k} \eqdef \, 
\text{ $k$-th rebirth (or branching) time of particle $n$},$$ with the convention $\tau_{n,0} =0$.
\item(Surviving particles) $X_{\tau_j}^{n,-}$, for $n\in\{1,\dots,N\}$, denotes the state of particle $n$ \emph{after} the last jump on $\partial$ of the last killed particle at $\tau_j$, but \emph{before} the resampling. We also denote %\arnaud{ajouter la def de Fred sur $X_{\tau_j}^{n,-}$}
%$$
%\Alive_{n,k} \eqdef \set{ \text{$(N-K)$ particles that are not resampled at time } \tau_{n,k} },
%$$
%and
\begin{align*}
\Alive_{j} &\eqdef \set{ \text{$(N-K)$ particles that are not resampled at time } \tau_{j} }\\
&:= \left\{ n\in\{1,\dots,N\}\ \text{s.t.}\ X_{\tau_j}^{n,-}\neq\partial \right\}.
\end{align*}
We may also use the individual indexation, for instance $\Alive_{n,k}$ is the set of particles that are not resampled at time $\tau_{n,k}$.
\item($\sigma$-field before resampling) For each branching time $\tau_j$ we also define $\calF_{\tau_j}^-\eqdef \calF_{\tau_j^-}\vee \sigma(X_{\tau_j}^{n,-}, n\in\{1,\dots,N\})$. We obviously have $\calF_{\tau_{j^-}}\subset \calF_{\tau_j}^- \subset \calF_{\tau_j}$.
\end{itemize}
\end{Def}

\begin{Rem}
\begin{enumerate}
\item In general, if $\k \geq 2$, the branching time $\tau_{n,k}$ is different from the $k$-th \emph{killing} time of particle $n$. However, note that the latter belongs to the time interval $(\tau_{n,k-1},\tau_{n,k}]$, the upper-bound being reached if $n$ is the last particle of $\set{1, \ldots, N} \setminus \Alive_{n,k}$ to be killed.
\item In the previous definition, the word ``Alive'' in ``$\Alive_{j}$'' is a slight abuse of terminology. Indeed, if for example $j=1$, all particles are alive at time $\tau_1$. At time $\tau_1^-$, $(K-1)$ are dead (i.e., equal to $\partial$) and $\tau_1$ is the $K$-th killing date, at which all of the  $K$ ``dead'' particles are instantaneously resampled, i.e., branched on the $(N-K)$ ``alive'' ones. This is illustrated on Figure \ref{fv4}. As we will see, there is exactly one particle for which $X_{\tau_j}^{n,-}$ is not equal to $X_{\tau_j^-}^{n}$, namely the particle which is killed at time $\tau_j$.
\end{enumerate}
\end{Rem}

This section builds upon the same martingale representation as in~\cite{v14}. Namely, we decompose the process $t \mapsto \gamma^N_t \p{Q}$ into the martingale contributions of the Markovian evolution of particle $n$ between branchings $k$ an $k+1$, which will be denoted $t \mapsto \M^{n,k}_t$, and the martingale contributions of the $k$-th branching of particle $n$, which will be denoted $t \mapsto \calM^{n,k}_t$.\medskip

\begin{Rem}
Throughout the paper, all the local martingales are local with respect to the sequence of stopping times $(\tau_j)_{j\geq 1}$. As required, this sequence of stopping times satisfies $\lim_{j\to\infty}\tau_j>T$ almost surely by Assumption \ref{ass:B}.
\end{Rem}

If $\widetilde{X}_t$ is any particle evolving according to the dynamics of the underlying Markov process for (and only for) $t < \tau_\partial$, then it is still true that $Q^{T-t}(\ph)(\widetilde{X}_t)\un_{t < \tau_\partial}$ is a martingale. As a consequence, for any $n\in\{1,\dots,N\}$ and any $k\ge 1$, Doob's optional sampling theorem ensures that, by construction of the particle system, the process
\begin{align}\label{lmdkcjm}
\M_t^{n,k} \eqdef \Big(\un_{t<\tau_{n,k}}\L_t^n-\L_{\tau_{n,k-1}}^n\Big)\un_{t\geq\tau_{n,k-1}}=
\begin{cases}
\dps 0 & \text{if } t < \tau_{n,k-1}  \\
\dps {\L^n_{t}} - {\L^n_{\tau_{n,k-1}}} &  \text{if }  \tau_{n,k-1} \leq t < \tau_{n,k} \\
\dps - {\L^n_{\tau_{n,k-1}}} &  \text{if }  \tau_{n,k} \leq t
\end{cases}
\end{align}
is a bounded martingale. Accordingly, under Assumption \ref{ass:B}, the processes
\begin{align}
& \M_t^{n} \eqdef\sum_{k = 1}^\infty \M_t^{n,k}=\L^n_t-\sum_{0 \leq \tau_{n,k}\leq t }\L^n_{\tau_{n,k}},\label{mbbtn}\\
& \M_t \eqdef\frac{1}{\sqrt{N}}\sum_{n=1}^N \M_t^{n}=\sqrt{N}\left(\L_t-\sum_{0 \leq \tau_{n,k}\leq t }\L_{\tau_{n,k}}\right),\label{mbbt}
\end{align}
are local martingales. The scaling by a $1/\sqrt{N}$ factor in the definition of $\M$ is there to ensure that the variance of the latter is of order $1$ in the large population limit.  \medskip

For any $n\in\{1,\dots,N\}$ and any $k\ge 1$, we also consider the process
\begin{equation}\label{aicj}
\calM_t^{n,k} \eqdef
\Big(1-\frac{\k}{N}\Big)\Big(\L^n_{\tau_{n,k}}-\frac1{N-\k} \sum_{m \in \Alive_{n,k}}\L^m_{\tau_{n,k}}\Big) 
\un_{t \geq \tau_{n,k}}. %%% la suite est fausse = \p{ \L^n_{\tau_{n,k}} - \L_{\tau_{n,k}} } \un_{t \geq \tau_{n,k}},
\end{equation}
By Lemma~\ref{albcios}, this is a piecewise constant martingale with a single jump at $t=\tau_{n,k}$, and it is clearly bounded by $2 \norm{\ph}_\infty$. Then, under Assumption \ref{ass:B}, the processes
\begin{align*}
&\calM_t^{n}\eqdef\sum_{k=1}^\infty\calM_t^{n,k}=\sum_{0 \leq \tau_{n,k}\leq t}\left(\L^n_{\tau_{n,k}} - \L_{\tau_{n,k}}\right),\\
& \calM_t\eqdef\frac{1}{\sqrt{N}}\sum_{n=1}^N\calM_t^n,
\end{align*}
are also local martingales. Again, the scaling by $1/\sqrt{N}$ is chosen to ensure that the variance of the latter is of order $1$ in the large population limit.

%{\tiny
%
%Recalling the notation $$
%\calN^{n}_t\eqdef   \sum_{k\geq 1} \un_{\tau_{n,k}\leq t}
%$$
%for the number of branchings of particle $n$ before time $t$ and  
%$$
%\calN_t\eqdef   \frac1N \sum_{n=1}^N \calN^{n}_t = \frac1N \sum_{j\geq 1} \un_{\tau_{j}\leq t}
%$$
%the total number of branchings per particle before time $t$, \eqref{lmdkcjm} and~\eqref{aicj} respectively implies for each $ 1 \leq n \leq N$ that %
%\begin{align}
%& d\M^n_t= d \L^n_t-\L_{t}^n d \calN^n_t\label{dmt}\\
%&d\calM^n_t =\big(\L^n_{t}- \L_{t} \big) d \calN^n_t,\label{dcalm}
%\end{align}
%so that in the sum yields
%\begin{align}\label{eq:L}
%& d\M_t+d \calM_t = \sqrt{N}\p{d \L_t-\L_{t}  \,  d \calN_t}.
%\end{align}
%Let us emphasize that, in the above equations, $\L_t=\L_{t^+}$ since the process $\L_t$ is right-continuous. \medskip
%}

\begin{Lem}[About the jumps of the martingales]
\label{lem:jumps} 
Under Assumption~\ref{ass:A}: 
\begin{enumerate}[(i)]
\item For each $n$, the jumps of $\calM_t^n$ only happen at branching times, more precisely at times $\tau_{n,k}$ for $k \geq 1$.
\item For all $j\geq 1$, one has $\Delta \M^n_{\tau_j} = 0$ unless $n$ is the only particle in $\set{1, \ldots, N} \setminus \Alive_j$ that is killed exactly at time $\tau_j$. In any case, one has $$\Delta \M^n_{\tau_j} = - \un_{n \notin \Alive_j} \L^n_{\tau_j^-}.$$
\item If $m\neq n$, the jumps of $\M_t^m$ and $\M_t^n$ never happen at the same time.
\end{enumerate}
\end{Lem}
\begin{proof}
(i) and (ii) are direct consequences of the definitions of $\calM$ and $\M$. For (iii), by construction, the jumps of the martingales $\M_t^{n}$ are included in the union of the set of the jumps of $\L_t^n$ and the set of the branching times $\tau_{n,k}$. Thus, Assumption~\ref{ass:A} ensures that for $n\neq m$, the jumps of $\M_t^n$ and $\M_t^m$ could happen at the same time only if the latter is a branching time. However, by (ii), $\M_t^{n}$ may jump only in the case where $n$ is the unique particle killed exactly at $\tau_{n,k}$, in which case $\M^m_t$ does not jump, hence (iii).
\end{proof}

%{\tiny
%
%The following rule will be useful throughout the paper.
%\begin{Lem} Recalling that $p^N_t = (1 - \frac1N)^{N\calN_t}$, it holds that
% \begin{equation}\label{eq:count_diff}
%d p^N_t=-p^N_{t^-} d \calN_t .
% \end{equation}
%\end{Lem}
%\begin{proof}
% One has $\Delta p^N_t = (1 - \frac1N)^{N\calN_t} - (1 - \frac1N)^{N\calN_t-1} = (1 - \frac1N)^{N\calN_t-1} \p{1-\frac1N -1 }$ while $\Delta \calN_t = \frac1N$ for $t = \tau_j$, $j \geq 1$. Hence the result.
%\end{proof}
%
%}

The upcoming result attests that the process $t\mapsto\gamma_t^N(Q)$ is indeed a martingale and details its decomposition.

\begin{Lem}\label{lem:decomp}
We have the decomposition
\begin{equation}\label{eq:decomp}
\gamma_t^N(Q) = \gamma_0^N(Q) +  \frac{1}{\sqrt{N}}\int_0^t \rho^N_{u^-} \p{ d\M_u + d\calM_u}.
\end{equation}
\end{Lem}

\begin{proof}
Considering (\ref{ljsnclnjc}) and (\ref{qlsjnclscn}), an integration by parts yields
$$\gamma_t^N(Q)=\rho_t^N\eta_t^N(Q^{T-t}(\ph))=\rho_t^N\L_t= \gamma_0^N(Q)+\int_0^t(\rho_{u^-}^Nd\L_u+\L_ud\rho_{u}^N),$$
where
$$\rho_t^N= \left(1-\tfrac{\k}{N}\right)^{\calB_t}= \left(1-\tfrac{\k}{N}\right)^{\sum_{j=1}^\infty\un_{\tau_j\leq t}}.$$
Hence, our goal is to show that 
$$\rho_{u^-}^Nd\L_u+\L_ud\rho_{u}^N =\frac{1}{\sqrt{N}}\rho^N_{u^-} \p{ d\M_u + d\calM_u}.$$
Between the branching times $\tau_j$, the result is obviously true since $\rho^N$ and $\calM$ are constant processes and, by (\ref{mbbt}), $d\M_u=\sqrt{N}d\L_u$.\medskip

 At branching time $\tau_j$, on the one hand, we get  by definition
 $$\Delta \gamma_{\tau_j}^N(Q) = \rho^N_{\tau_j^-} \frac1N \sum_{n=1}^N\p{ (1-\k/N) \L^n_{\tau_j}    - \L^n_{\tau_j^-}  }= \rho^N_{\tau_j^-}\left[(1-\k/N) \L_{\tau_j}    - \L_{\tau_j^-}\right].$$
 On the other hand, Lemma~\ref{lem:jumps} gives
$$ \frac{1}{\sqrt{N}} \Delta \M_{\tau_j}  = - \frac{1}{N} \sum_{n \notin \Alive_j} \L^n_{\tau_j^-}.$$
Note that, in the latter, only the particle that is killed at time $\tau_j$ contributes to the sum. Moreover, in a similar fashion, 
 \begin{align}\label{lajcaljcljsc}
 \frac{1}{\sqrt{N}} \Delta \calM_{\tau_j} & = \frac1N (1-\k/N) \p{ \sum_{n \notin \Alive_j} \L^n_{\tau_j}   - \frac{\k}{N-\k} \sum_{n \in \Alive_j} \L^n_{\tau_j}}.
 \end{align}
 By Assumption~\ref{ass:A}, $\L^n$ does not jump at $\tau_j$ if $n \in  \Alive_j$, so that
 \begin{align*}
 \frac{1}{\sqrt{N}} \Delta \calM_{\tau_j} & = \frac1N (1-\k/N) \p{ \sum_{n} \L^n_{\tau_j} - \sum_{n \in \Alive_j} \L^n_{\tau_j^-}   - \frac{\k}{N-\k} \sum_{n \in \Alive_j} \L^n_{\tau_j^-} }  \\
 & = \frac1N (1-\k/N) \sum_{n} \L^n_{\tau_j} - \frac1N \sum_{n \in \Alive_j} \L^n_{\tau_j^-}\\
 & = (1-\k/N) \L_{\tau_j} - \frac1N \sum_{n \in \Alive_j} \L^n_{\tau_j^-}.
 \end{align*}
 Computing the sum $\Delta \M_{\tau_j}+\Delta \calM_{\tau_j}$  yields the result.
\end{proof}

%
%{\tiny
%
%\begin{proof}
%Recalling that $p^N_t$ is a piecewise constant process, one has by plain integration by parts
%\begin{align*}
%\gamma_t^N(Q)= p^N_t \L_t=\gamma_0^N(Q)+ \int_0^t \p{ p^N_{u^-} d\L_u+\L^n_{u} d p^N_u },
%\end{align*}
%where we emphasize that in the above equation, the last integrand is indeed $\L_u=\L_{u^+}$.
%Besides, by~\eqref{eq:count_diff}, we are led to
%$$
%\gamma_t^N(Q) - \gamma_0^N(Q) = \int_0^t p^N_{u^-} \p{ d \L_u- \L_{u} d \calN_u }.
%$$
%The result is then a direct consequence of~\eqref{eq:L}.
%\end{proof}
%\begin{Rem}
%Since $\gamma_T(\ph)=\gamma_0( Q^T  \ph)$, this implies the unbiasedness property $\E\b
% {\gamma^N_T(\ph)}=\gamma_T(\ph)$ for all $N \geq 2$. In particular, the case $\ph=\un_F$ gives $\E\b{p^N_T(\ph)}= p_T > 0$.
%\end{Rem}
%
%
% }

%
\subsection{Quadratic variation analysis}\label{sec:martquad}

The remarkable fact is that the $2N$ martingales $\set{\M^n_t,\calM^m_t}_{1 \leq n,m \leq N}$ are mutually orthogonal. We recall that two local martingales are orthogonal if and only if their quadratic covariation is again a local martingale.
 
 \begin{Lem} \label{Lem:quad} Under Assumptions~\ref{ass:A} and~\ref{ass:B}, the $N^2$ local martingales $\set{\M^n_t,\calM^m_t}_{1 \leq n,m \leq N}$ are mutually orthogonal. As a consequence,
$$\b{ \calM,\calM}_t = \frac{1}{N}\sum_{n=1}^N\b{ \calM^n,\calM^n}_t + \text{local martingale},$$
$\b{ \M,\cal M}_t$ is a local martingale, and
\begin{align*} 
\b{ \M,\M}_t = \frac{1}{N}\sum_{n=1}^N\b{ \M^n,\M^n}_t + \text{local martingale}. 
\end{align*}
In what follows, we adopt the notation
$$
\A_t \eqdef \frac{1}{N}\sum_{n=1}^N\b{ \M^n,\M^n}_t .
$$
The jumps of $\A$ are controlled by
\begin{align*}%\label{ajump}
\Delta\A_t \leq \frac{\|\ph\|_\infty^2}{N}.
\end{align*}
\end{Lem}
\begin{proof}
i) Let us show that $\b{\calM^n,\calM^m}_t$ is a local martingale for $n \neq m$. Indeed, $\calM^n$ is piecewise constant outside the branching times so that
$$
\b{\calM^n,\calM^m}_t = \sum_{j} \un_{t \geq \tau_j} \Delta \calM^n_{\tau_j} \Delta\calM^m_{\tau_j},
$$
where, by definition, $\Delta \calM^n_{\tau_j} = 0$ if $n \in \Alive_{j}$ while, otherwise, %(note that Lemma-\ref{lem:jumps} implies $\L^m_{\tau_{j}^-} = \L^m_{\tau_{j}}$ if $m \in \Alive_{j}$)
$$
\Delta \calM^n_{\tau_j} = \Big(1-\frac{\k}{N}\Big)\Big(\L^n_{\tau_{j}}-\frac1{N-\k} \sum_{m \in \Alive_{j}}\L^m_{\tau_{j}^-}\Big)
$$
which by construction has zero average conditionally on $\calF_{\tau_j}^-$ (uniform resampling among the living particles). In the same way, $\Delta \calM^n_{\tau_j}$ and $\Delta \calM^m_{\tau_j}$ are independent for $n \neq m$ conditionally on $\calF_{\tau_j}^-$, 
%$\calF_{\tau_j}^-\eqdef  \calF_{\tau_j^-}\vee \sigma(X_{\tau_j}^{n,-}, n\in\{1,\dots,N\})$ 
by conditional independence of the resampling of killed particles. As a consequence,
$$
\E \b{\left. \Delta \calM^n_{\tau_j} \Delta \calM^m_{\tau_j}  \right| \calF_{\tau_j^-} } = \E\left[ \left. \E \b{\left. \Delta \calM^n_{\tau_j} \Delta \calM^m_{\tau_j}  \right| \calF_{\tau_j}^- } \right| \calF_{\tau_j^-} \right]=0,
$$
and Lemma~\ref{albcios} allows us to conclude the proof of Step~i). \medskip

ii) We claim that $\b{\M^n,\calM^m}_t$ is a local martingale. Since $\calM^m$ is a pure jump martingale that only jumps at branching times, we have
$$ \b{ \M^{n},\calM^{m}}_t = \sum_{j} \un_{t \geq \tau_j}  \Delta\M^n_{\tau_j} \Delta \calM^m_{\tau_j}.$$ 
As explained in Lemma \ref{lem:jumps}, $\Delta \M^n_{\tau_j}$ can be non zero only when particle $n$ is the single particle killed exactly at time $\tau_j$. Specifically, we have
$$
\Delta \M^n_{\tau_j} = - \un_{n \notin \Alive_j} \L^n_{\tau_j^-},
$$
which is measurable with respect to $\calF_{\tau_j}^-$. Consequently,
$$
\E \b{ \left. \Delta\M^n_{\tau_j} \Delta \calM^m_{\tau_j}  \right| \calF_{\tau_j^-} } = \E\left[ \left. - \un_{n \notin \Alive_j} \L^n_{\tau_j^-}\E \b{ \left. \Delta \calM^m_{\tau_j}  \right| \calF_{\tau_j}^- }     \right|    \calF_{\tau_j^-}   \right] = 0,
$$
and Lemma~\ref{albcios} concludes the proof of Step~ii). \medskip

%
%{\tiny
%$n \neq m$, the martingales $\calM^n_t$ and $ \M^m_t$ do not vary at the same times, so that $\b{\calM^n,\M^m}_t=0$ and the two martingales are {\it a fortiori} orthogonal.\medskip
%Moreover, since $\calM^n$ is a pure jump martingale, we have by definition of $\M_t^{n}$
%$$d\b{ \M^{n},\calM^{n}}_t =  \Delta\M^n_t d\calM^n_t=- \L_{t^-}^n d\calM^n_t,$$ 
%which defines a martingale, so that $\M^n$ and $\calM^n$ are orthogonal.\medskip 
%}

iii) We claim that the product  $\M^n \M^m$ is a local martingale for $n\neq m$. The proof is similar to the one of Lemma 3.9 in~\cite{cdgr2}. Let us just briefly mention that it relies on the following facts: $\M^n$ and $\M^m$ are by construction independent between branching times (conditionally on the past), and never jump simultaneously, even at branching times by Lemma \ref{lem:jumps}.\medskip

%{\tiny
%
%we claim that the product $\M^m\M^n$ is a martingale, implying the orthogonality.
%Indeed, for a given $s \in [0,T]$, let us define $\sigma_i \eqdef (\tau_i \wedge T) \vee s$ the stopping time in $[s,T]$ closest to the $i$-th branching time. For any $i\ge 1$, conditional to $\mathcal F_{\sigma_{i-1}}$, $(\M^n_t1_{t<\sigma_i})_{t\ge 0}$ 
%and $(\M^m_t1_{t<\sigma_i})_{t\ge 0}$ are by construction independent, 
%hence
%\begin{align*}
%\E\left[\left.\M^m_{\sigma_i^-}\M^n_{\sigma_i^-}\right|\mathcal F_{\sigma_{i-1}}\right]=\M^m_{\sigma_{i-1}}\M^n_{\sigma_{i-1}}.
%\end{align*}
%In addition, since the martingales $\M^m$ and $\M^n$ do not jump simultaneously, it yields $\M^m_{\sigma_i}\M^n_{\sigma_i} = \M^m_{\sigma_i}\M^n_{\sigma^-_i}+\M^m_{\sigma^-_i}\M^n_{\sigma_i} - \M^m_{\sigma^-_i}\M^n_{\sigma^-_i}$, so that
%\begin{align*}
% \E[\M^m_{\sigma_i}\M^n_{\sigma_i}|\mathcal F_{\sigma_i^-}]
%%
%& =\E[\M^n_{\sigma_i^-}(\M^m_{\sigma_i}-\M^m_{\sigma_i^-})+\M^m_{\sigma_i^-}(\M^n_{\sigma_i}-\M^n_{\sigma_i^-})+\M^m_{\sigma_i^-}\M^n_{\sigma_i^-}|\mathcal F_{\sigma_i^-}]\\
%&= \M^m_{\sigma_i^-}\M^n_{\sigma_i^-},
%\end{align*}
%and combining these equations gives
%\begin{align*}
% \E\left[\left.\M^m_{\sigma_i}\M^n_{\sigma_i}\right|\mathcal F_{\sigma_{i-1}}\right]= \M^m_{\sigma_{i-1}}\M^n_{\sigma_{i-1}}.
%\end{align*}
%By iterating on $i \ge 1$ and taking into account that $\sigma_0 = s$ and $\lim_{i \to +\infty} \sigma_i = T$, we obtain
%\begin{align*}
% \E\left[\left.\M^m_T\M^n_T\right|\mathcal F_{s}\right]=  \M^m_{s}\M^n_{s},
%\end{align*}
%which shows the claimed result. \medskip
%}

For the last point, Assumption~\ref{ass:A} guarantees that
$$\Delta\A_t=\frac{1}{N}\max_{1\leq n\leq N}\Delta[\M^n,\M^n]_t=\frac{1}{N}\max_{1\leq n\leq N}\left(\Delta\M_t^n\right)^2,$$
and the indicated result is now a direct consequence of (\ref{lmdkcjm}) and (\ref{mbbtn}).
\end{proof}

Our next objective is to calculate the quadratic variation $\frac1N \sum_n \b{\calM^n, \calM^n}$. Following~\eqref{eq:M_t}, and remarking that $\L^n_{\tau_{j}}=\L^n_{\tau_{j}^-}$ for all $n \in \Alive_j$ by Lemma~\ref{lem:jumps}, we also adopt  the upcoming notation.

\begin{Not}\label{not:not_var_emp} The empirical distribution of the particles that are ``alive'' at branching time $\tau_j$ is denoted
$$
\eta^N_{\Alive_j} \eqdef \frac{1}{N-\k} \sum_{n \in \Alive_j} \delta_{X^n_{\tau_j}} .%=\frac{N}{N-K}\,\eta^N_{\tau^-_j}(\un_F \ \,  .).
$$
Accordingly, we have
\begin{align*}
 \Var_{\eta^N_{\Alive_j}}(Q) = \Var_{\eta^N_{\Alive_j}}(Q^{T-\tau_j}(\ph))
  = \frac{1}{N-K} \sum_{n \in \Alive_j}  \left[\L_{\tau_j^-}^n - \tfrac{1}{N-K} \sum_{m \in \Alive_j}  \L_{\tau_j^-}^m \right]^2.
\end{align*}
Mutatis mutandis, $\Var_{\eta^N_{\Alive_j}}(Q^2)$ is defined in the same manner.
\end{Not}

Note that $\Var_{\eta^N_{\Alive_j}}(Q)$ and $\Var_{\eta^N_{\Alive_j}}(Q^2)$ are measurable with respect to $\calF_{\tau_j}^-$.

\begin{Lem}\label{lemab}
There exists a piecewise constant local martingale $\tcalM_t$ with jumps at branching times, such that
  \begin{align*}%
\b{\calM,\calM}_t
&= \p{1-\frac{\k}{N}}^2 \frac{\k}{N} \sum_{j \geq 1} \un_{\tau_j \leq t} \Var_{\eta^N_{\Alive_j}}(Q) + \frac{1}{\sqrt{N}} \tcalM_t. %\label{lemab1}
 \end{align*}

%{\tiny
% 
% There exist a piecewise constant local martingale $\tcalM_t$ and a piecewise constant process $\calR_t$, 
%both with jumps at branching times, such that
%  \begin{align}%
%d\b{\calM,\calM}_t
%%
%&=\Var_{\eta^{N}_{t^-}}(Q) d \calN_t + \frac1N d\calR_t + \frac{1}{\sqrt{N}}d\tcalM_t, \label{lemab1}
% \end{align}
% with the following estimate
% \begin{align}\label{lemab2}
%|\Delta \calR_{t}|\le  \frac{14 \norm{\ph}_{\infty}^2}{N}.
% \end{align}
% 
%}

Since $\Var_{\eta^N_{\Alive_j}}(Q) \leq 2 \norm{\ph}_{\infty}^2$, we deduce that
 \begin{align}\label{mrondcroch}
  d \b{\calM,\calM}_t \le  2 \p{1-\frac{\k}{N}}^2 \frac{\k}{N} \|\ph\|_\infty^2 d\calB_t + \text{local martingale}.
  \end{align}
\end{Lem}

\begin{proof}
Considering the orthogonality property in Lemma \ref{Lem:quad}, and taking into account that the martingales $\calM^{n,k}$ are piecewise constant with a single jump at time $\tau_{n,k}$, we have
$$\b{\calM,\calM}_t
=\frac1N\sum_{n=1}^N\sum_{j=1}^{+\infty} \p{\Delta \calM^{n}_{\tau_{j}}}^2\un_{t\geq\tau_{j}}.$$

We can then define 
\begin{align*}
&\tcalM_t:=\frac{1}{\sqrt N}\sum_{n=1}^N\sum_{j=1}^{+\infty}\p{\big(\Delta  \calM^{n}_{\tau_{j}}\big)^2 
-  \E\left[\big(\Delta  \calM^{n}_{\tau_{j}}\big)^2  \big| \calF_{\tau_{j}}^- \right] }\un_{t \geq \tau_{j}}
 \end{align*}
 which is indeed a local martingale by Lemma~\ref{lem:jumps} and Lemma \ref{albcios}. Recall that $\Delta \calM^{n}_{\tau_{j}} = 0$ if $n \in \Alive_j$. Otherwise, by construction of the branching rule, we obtain

 \begin{align*}
\E\left[\big( \Delta \calM^{n}_{\tau_{j}}\big)^2 \big| \calF_{\tau_{j}}^- \right]
&=\p{1-\tfrac{\k}{N}}^2 \tfrac1{ N-\k } \sum_{m\in \Alive_j}\left(\L^m_{\tau_{j}^-}- \tfrac{1}{N-\k} \sum_{l \in \Alive_j}  \L^l_{\tau_{j}^-}\right)^2 \\
& = \p{1-\tfrac{\k}{N}}^2 \Var_{\eta^N_{\Alive_j}}(Q) ,
\end{align*}
which is independent of the choice of the resampled particle $n$. Since there are exactly $\k$ resampled particles at time $\tau_j$, this yields

$$
\frac1N \sum_{n=1}^N \E\left[\big( \Delta \calM^{n}_{\tau_{j}}\big)^2 \big| \calF_{\tau_{j}}^- \right]
= \tfrac{\k}{N} \p{1-\tfrac{\k}{N}}^2 \Var_{\eta^N_{\Alive_j}}(Q) , 
$$
hence the result.
\end{proof}

The next lemma is a crucial step of the analysis. It relates the quadratic variation of the local martingale $t \mapsto \M_t$ - given, up to a martingale additive term, by the increasing process $t \mapsto \A_t$ defined in Lemma~\ref{lem:decomp} -, with the process $t \mapsto \gamma^N_t(Q^2)$. This leads to estimates on $\A_t$. This idea is inspired by the fact that, by definition of the quadratic variation, and for any Markov process $X$, the process $t \mapsto \b{Q^{T-t}(\ph)(X_t)}^2$ equals the quadratic variation of the martingale $t \mapsto Q^{T-t}(\ph)(X_t)$ up to a martingale additive term.

\begin{Lem}\label{lemgQ2}
  One has the decomposition
  \begin{align}\label{lemgQ21}
  d\gamma^N_t(Q^2) =  \rho_{t^-}^N d\A_t+ \frac{1}{\sqrt{N}} \rho_{t^-}^N  d\tM_t ,
  \end{align}
  where $(\tM_t)_{t \geq 0}$ is a local martingale satisfying
  \begin{align}
  \E\left[\int_0^t \rho_{u^-}^Nd[\tM,\tM]_u\right]\le  6 \|\ph\|_\infty^4.  \label{lemgQ24}
  \end{align}
\end{Lem}

\begin{Rem}
\eqref{lemgQ21} implies that
  \begin{align*}%\label{major}
  \E\left[\int_0^t \rho_{s^-}^N d \A_s\right]= \E \b{ \gamma^N_t(Q^2) - \gamma^N_0(Q^2)} \leq \norm{\ph}_\infty^2.
  \end{align*}
\end{Rem}

\begin{proof}
Since $\calB_t$ denotes the number of branching times until time $t$, it comes
  $$
  d \rho_t^N = - \frac{\k}{N} \rho_{t^-}^N d\calB_t.
  $$
  If $\calB_t^n$ denotes the number of branching times of particle $n$ until time $t$, we have 
  $$
  \calB_t = \frac1K \sum_{n=1}^N \calB_t^n
  $$
  since, according to Lemma~\ref{lem:jumps}, exactly $\k$ particles are resampled at each branching time. We will now prove (\ref{lemgQ21}) and calculate the martingale part $\tM$. Differentiating 
 $$\gamma^N_t(Q^2) \eqdef \rho_t^N \frac1N \sum_{n=1}^N (\L_t^n)^2$$ 
 yields
  \begin{align}\label{eq:step_1}
  d \gamma^N_t(Q^2) & = \frac1N \sum_{n=1}^N  \rho^N_{t^-} d \p{ (\L_t^n)^2 }+(\L_{t}^n)^2 d \rho^N_t\nonumber \\
  &=\frac1N\sum_{n=1}^N  \rho^N_{t^-} \left(d \p{ (\L_t^n)^2 }-\frac{\k}{N} (\L_{t}^n)^2 d\calB_t\right).
    \end{align}
    
  Next, we claim that 
  \begin{equation}\label{eq:step_2}
   d(\L_t^n)^2 -  (\L_{t}^n\big)^2 d\calB^n_t = d[\M^n,\M^n]_t + 2\L^n_{t^-}d\M_t^n .
  \end{equation}
  First, by definition of $\M^n$ (see~\eqref{mbbtn}), $d\M_t^n=d\L_t^n -\L_{t}^n d \calB^n_t$, so that the bilinearity of the quadratic variation gives
  \begin{align*}
  d[\M^n,\M^n]_t
  &=d[\L^n,\L^n]_t+\big(\L_{t}^n\big)^2d\calB^n_t-2 d \Big [ \int \L^n d \calB^n,\L^n \Big ]_t\\
  &=d[\L^n,\L^n]_t+\big(\L_{t}^n\big)^2d\calB^n_t-2(\Delta \L^n_t)\, \L_{t}^n d\calB^n_t\\
  &=d[\L^n,\L^n]_t+\L_{t}^n\big(2\L_{t^-}^n-\L_{t}^n\big)d\calB^n_t.
  \end{align*}
  Then, using again $ d\L_t^n = d\M_t^n +\L_{t}^n d \calB^n_t$, this yields
  \begin{align*}
  d(\L_t^n)^2
  &=2\L^n_{t^-}d\L_t^n+d[\L^n,\L^n]_t\\
  &=\Big(2\L^n_{t^-}d\M_t^n +2\L^n_{t^-}\L_{t}^nd\calB^n_t\Big)+\Big(d[\M^n,\M^n]_t
  -\L_{t}^n\big(2\L_{t^-}^n-\L_{t}^n\big)d\calB^n_t\Big),%
  \end{align*}
 which immediately simplifies into~\eqref{eq:step_2}. Putting~\eqref{eq:step_1},~\eqref{eq:step_2}, and the very definition of $\A = \frac1N \sum_n [\M^n,\M^n]$ together, we obtain
  \begin{align*}
    d \gamma^N_t(Q^2)&=\rho_{t^-}^N d\A_t+\frac{\rho_{t^-}^N}{N}\sum_{n=1}^N \b{ (\L_{t}^n)^2  \left(d\calB_t^n-\frac{\k}{N}d\calB_t\right)+2\L^n_{t^-}d\M_t^n }.
  \end{align*}
  
  Now, by definition of the counting processes $\calB^n$ and $\calB$, 
  $$
  \sum_{n=1}^N (\L_{t}^n)^2  d\calB_t^n = \b{ \sum_{n \notin \Alive_t} (\L_{t}^n)^2} d\calB_t,
  $$
  where we have used the notation
  $$
  \Alive_t \eqdef \set{\text{particles that are not resampled at time $t$}}.
  $$
  As a consequence,
  \begin{align*}
  d \gamma^N_t(Q^2) =&\rho_{t^-}^N d\A_t+ \frac{\rho_{t^-}^N }{N} \b{  (1-\k/N) \sum_{n \notin \Alive_t} (\L_{t}^n)^2- \frac{\k}{N} \sum_{n \in \Alive_t}(\L_{t}^n)^2 } d\calB_t\\
 &  + \frac{ \rho_{t^-}^N }{N} \sum_{n=1}^N 2\L_{t^-}^n d\M_t^n.
  \end{align*}
  Hence we see that \eqref{lemgQ21} is satisfied with
  \begin{align}\label{beq}
  d\tM_t %
  &= \frac{1}{\sqrt{N}} J_t d\calB_t + \frac{1}{\sqrt{N}} \sum_{n=1}^N
  2\L_{t^-}^nd\M_t^n ,
  \end{align}
where we have defined 
$$J_t \eqdef \frac{1-\k/N}{\sqrt{N}}  \b{ \sum_{n \notin \Alive_t} (\L_{t}^n)^2- \frac{\k}{N-\k} \sum_{n \in \Alive_t}(\L_{t}^n)^2 }.$$
It is readily seen that
$$
\E\b{J_{\tau_j} \left| \calF_{\tau_j}^- \right.} = 0,
$$
so that, according to Lemma~\ref{albcios}, $\tM$ is indeed a local martingale. Using Notation~\ref{not:not_var_emp} and the fact that $\sup_{t\geq0}|\L_{t^-}^n|\leq\|\ph\|_\infty$, we also have 
\begin{align}\label{eq:var_J}
  \E\b{J_{\tau_j}^2\left| \calF_{\tau_j}^- \right. } & = \p{1-\k/N}^2 \frac{\k}{N} \Var_{\eta^N_{\Alive_j}} \p{ Q^2 } \nonumber \\
  & \leq 2 \p{1-\k/N}^2 \frac{\k}{N} \norm{\ph}_\infty^4 .
\end{align}

%{\tiny Note that, in the same fashion as $(\calM^n_t)_{t \geq 0}$, $ \p{ \int_0^t J^n_s d \calN_s^n }_{t \geq 0}$ is a local martingale since, for each $k\geq 1$, $ t \mapsto J^n_t \un_{t \geq \tau_{n,k}}$ is a bounded martingale by definition of the branching rule and Lemma~\ref{albcios}.}

We can now calculate the quadratic variation of $\tM$. In the same way as in Lemma~\ref{lem:decomp}, the $(N+1)$ local martingales 
$$\set{\p{\int_0^t\L_{s^-}^m d \M_s^m }_{t \geq 0}, \, 1\leq n \leq N; \, \int J_s d\calB_s }$$ are all orthogonal to each other. Indeed, by Lemma~\ref{lem:decomp}, $ [\M^n,\M^m]$ is a martingale for any pair $n \neq m$. The only new point to check (using again Lemma~\ref{lem:jumps}) is that the quadratic covariation 
$$d \b{\int \L_{s^-}^n d \M_s^n ,\int J_s d \calB_s  }_{t} = - \p{\L_{t^-}^n}^2 J_t d\calB_t$$ 
is indeed a local martingale, which is again a consequence of $\E[J_{\tau_{j}} | \calF_{\tau_{j}}^-]=0$ and Lemma~\ref{albcios}. To establish~\eqref{lemgQ24}, we apply Itô's isometry to~\eqref{beq} and use orthogonality to obtain
\begin{align*}
  \E\int_0^t \rho_{u^-}^Nd[\tM,\tM]_u & = \E \b{ \int_0^t \rho_{u^-}^N \p{J_u}^2 d\calB_u  + \frac{4}{N}\sum_{n=1}^N \int_0^t \rho_{u^-}^N (\L_{t^-}^n)^2d[\M^n,\M^n]_u}.
\end{align*}
On the one hand, using~\eqref{eq:var_J}, we get
\begin{align*}
\E\left[\int_0^t \rho^N_{u^-}(J_u)^2 d \calB_u\right]  & = \E\left[\sum_j \one_{ \tau_j \leq t }\rho^N_{\tau_j^-}\E \b{ J_{\tau_j}^2 \left| \calF_{\tau_j}^-\right.}\right] \\
& = \sum_{j \geq 1} \one_{ \tau_j \leq t } (1-\k/N)^{j+1} \frac{\k}{N} \E \b{ \V_{\eta^N_{\Alive_j}}(Q^2)}  \\
& \leq 2 (1-\k/N)^{2}\|\ph\|_\infty^4,
\end{align*}
while, on the other hand, (\ref{lemgQ24}) implies
\begin{align*}
\frac{4}{N}\sum_{n=1}^N \int_0^t \rho_{u^-}^N (\L_{t^-}^n)^2d[\M^n,\M^n]_u 
  &\le 4\|\ph\|_\infty^2 \E\left[\int_0^t \rho_{u^-}^N d\A_u\right]\\
  &\le 4\|\ph\|_\infty^2 \E\left[\gamma_t^N(Q^2)-\gamma_0^N(Q^2)\right]\\
  &\le 4\|\ph\|_\infty^4.
\end{align*}
 
 Combining both inequalities yields the result.
 
 \qedhere

 \end{proof}

% {\tiny 
%  One has as well
%  \begin{align}
%  |\Delta\tM_u|\le  \frac{5\|\ph\|_\infty^2}{\sqrt{N}}.\label{lemgQ25}
%  \end{align}
% Recall $\sup_{t\geq 0}|\Delta\M_t^n|\leq 2\|\ph\|_\infty$.
%
%
%   In order to obtain~\eqref{lemgQ25}, consider (\ref{beq}) and recall from Assumption~\ref{ass:Ap} that, for $n\neq m$, $\Delta\calN_t^n\Delta\calN_t^m=0$ and $\Delta\M_t^n\Delta\M_t^m=0$. We then deduce that
% $$|\Delta\tM_t|\le \frac{1}{\sqrt{N}}\left(\sup_{t\geq 0}|J_t^n|+2\sup_{t\geq0}|\L_{t^-}^n|\sup_{t\geq 0}|\Delta\M_t^n|\right)\le \frac{5\|\ph\|_\infty^2}{\sqrt{N}}.$$
% 
% }

%
\subsection{$\L^2$-estimate}

The convergence of $\gamma^N_T(\ph)$  to $\gamma_T(\ph)$ when $N$ goes to infinity is now a straightforward consequence of the previous results. This kind of estimate was already noticed by Villemonais in~\cite{v14} for classical Fleming-Viot particle systems (i.e., in the case where $K=1$).

\begin{Pro}\label{pro:estimate}
For any $\ph \in {\calD}$ and any $\k \in [1,N]$, we have
\begin{align*}
\E \b{ \p{ \gamma^N_T(\ph) - \gamma_T(\ph) }^2 } \leq \frac{4 \norm{\ph}_\infty^2}{N}.
\end{align*}
\end{Pro}

\begin{proof} 
Thanks to Lemma \ref{lem:decomp} and the fact that $\gamma_T(\ph)=\gamma_0( Q^T  \ph)$, we have the orthogonal decomposition
$$\gamma^N_T(\ph) - \gamma_T(\ph)  =\frac{1}{\sqrt{N}} \int_0^T \rho^N_{t^-}\ d\M_t + \frac{1}{\sqrt{N}} \int_0^T \rho^N_{t^-}\ d\calM_t +  \gamma^N_0 (Q^T\ph) - \gamma_0( Q^T  \ph),$$
and we can upper-bound the contribution of each term to the total variance.

(i)~Initial condition. Since $\gamma_0=\eta_0$ and $\gamma_0^N=\eta_0^N$, we have
$$\E\b{ \p{\gamma^N_0 (Q^T\ph) - \gamma_0 (Q^T  \ph)}^2  }=\tfrac1N\Var_{\eta_0}(Q^T(\ph)(X))
\leq \tfrac1N\|Q^T(\ph)\|_\infty^2\leq \tfrac1N\|\ph\|_\infty^2.$$

(ii)~$\calM$-terms. Using Itô's isometry and (\ref{mrondcroch}), we obtain
\begin{align*}
\E\left[\left(\int_0^T \rho_{t^-}^N d\calM_t\right)^2\right]
&=\E\left[\int_0^T \big(\rho_{t^-}^N\big)^2 d[\calM,\calM]_t\right]\\
&\leq 2 \|\ph\|_\infty^2 \frac{\k}{N} \sum_{j=1}^{\infty} \p{1-\frac{\k}{N}}^{2j}\leq 2 \|\ph\|_\infty^2. 
\end{align*}

(iii)~$\M$-terms. In the same way, applying Itô's isometry and (\ref{lemgQ21}), we get
\begin{align*}
\E\left[\left(\int_0^T \rho_{t^-}^N d\M_t\right)^2\right]
&=\E\left[\int_0^T \big(\rho_{t^-}^N\big)^2 d[\M,\M]_t\right]\\
&\le\E\left[\int_0^T \rho_{t^-}^N d\A_t\right]=\E\left[\gamma_T^N(Q^2)\right]\leq \|\ph\|_\infty^2.
\end{align*}
\end{proof}

In particular, Proposition~\ref{pro:estimate} implies that for any $\ph$ in ${\cal D}$, $\gamma^N_t(\ph)$ converges in probability to $\gamma_t(\ph)$ when $N$ goes to infinity. Since we have assumed that $\one_F$ belongs to ${\cal D}$, the probability estimate $p_t^N=\gamma^N_t(\un_F)$ goes to its deterministic target $p_t=\gamma_t(\un_F)$ in probability. An interesting consequence is our first main result Proposition~\ref{pro:quant} that we can now justify.

\begin{proof}[{\bf Proof of Proposition~\ref{pro:quant}}]
% Using the Portmanteau theorem and an iterated H\"older inequality, it is sufficient to prove
% $$
% \P\p{ \tau_j \notin [t_j - \eps , t_j + \eps] } \to 0
% $$
% for each $j$ and each $\eps$. But by definition
% 
% $$
% \P\p{ \tau_j \notin [t_j - \eps , t_j + \eps] } \leq \P \p{ (1-\k/N)^j \leq p^N_{t_j + \eps} } + \P \p{ (1-\k/N)^j \geq p^N_{t_j - \eps} } 
% $$
% which converges to $0$ by the $\L^2$ estimate.
Fix $j \in [1,j_{\rm max}]$ and $\varepsilon>0$. The strict monotonicity assumption ensures that 
$$\delta_1=p_{t_j-\varepsilon}-p_{t_j}>0\hspace{1cm}\mathrm{and}\hspace{1cm}\delta_2=p_{t_j}-p_{t_j+\varepsilon}>0.$$ 
We have to prove that $ \P\p{ \tau_j \notin [t_j - \eps , t_j + \eps] }$ goes to zero when $N$ goes to infinity. Consider first the probability $\P\p{ \tau_j <t_j-\varepsilon}$. We have
$$
\{\tau_j <t_j-\varepsilon\}\subset\{p^N_{t_j-\varepsilon}\leq(1-K_N/N)^j\} \subset \{p^N_{t_j-\varepsilon}<\theta^j+\delta_1/2\},
$$
for $N$ large enough, using $1-K_N/N\to\theta$. 
From Proposition \ref{pro:estimate}, we know that $p^N_{t_j-\varepsilon}$ converges in probability to $p_{t_j-\varepsilon}=p_{t_j}+\delta_1=\theta^j+\delta_1$, which implies that $\P(p^N_{t_j-\varepsilon}<\theta^j+\delta_1/2)\to 0$ as $N$ goes to infinity. The term $\P\p{ \tau_j >t_j+\varepsilon}$ is treated similarly.\medskip 

For the last assertion, let $\delta=p_T-\theta^{\jmax+1}>0$. Then Proposition \ref{pro:estimate} implies that
$$
\P\left(\tau_{\jmax+1}\leq T\right)=\P(p_T^N\leq\theta^{\jmax+1})= \P\left(p_T^N\leq p_T-\delta\right) \xrightarrow[N \to + \infty]{} 0.
$$
\end{proof}
\begin{Rem}
\label{rem:niter}
An immediate consequence of Proposition~\ref{pro:quant} is that, if we denote by $\jmaxN$ the actual number of resamplings until final time $T$, we have, for all $\varepsilon>0$, 
$$
\P\left(|\jmaxN-\jmax|>\varepsilon\right)=\P\left(\jmaxN\neq\jmax\right) \xrightarrow[N \to + \infty]{} 0.
$$
\end{Rem}

%*************************    A COMPLETER . FIN MODIF MATHIAS 17 AOUT 2019 **************
%\subsection{Continuity of the limit}\label{sec:cont}
\subsection{Convergence of empirical measures at branching times}
As for (\ref{qlsjnclscn}), we denote
$$\gamma_t(Q)=\gamma_t(Q^{T-t}(\ph))\hspace{1cm}\mathrm{and}\hspace{1cm}\gamma_t(Q^2)=\gamma_t((Q^{T-t}(\ph))^2)$$
where, again, the parameters $T$ and $\ph$ are omitted.

\begin{Lem}
The function $t\mapsto\gamma_t(Q^2)$ is continuous on $0 \leq t \leq T$. 
\label{lem:cont}
\end{Lem}
\begin{proof}
First, by Assumption \ref{ass:A}, the distribution of the times at which the bounded martingale $M_t=Q^{T-t}(\varphi)(X_t)$ jumps is atomless, hence $M_t$ is almost surely continuous in $t$, and so is $M^2_t$. Second, fix $0\leq t\leq T$. By definition,
$$
\gamma_t(Q^2)=\E\left[\un_{\tau_\partial>t}(Q^{T-t}(\varphi)(X_t))^2\right]=\E\left[(Q^{T-t}(\varphi)(X_t))^2\right]=\E\left[M^2_{t}\right],
$$
and by dominated convergence, 
$$
\lim_{h\to 0} \gamma_{t+h}(Q^2)=\lim_{h\to 0} \E\left[M^2_{t+h}\right]=\E\left[\lim_{h\to 0}M^2_{t+h}\right]=\E\left[M^2_{t}\right]=\gamma_t(Q^2),
$$
which proves the continuity. 
\end{proof}

The proof of the CLT relies on the analysis of the convergence of the quadratic variation of the martingale $\gamma^N(Q)$ when $N$ goes to infinity. This requires to study the convergence of specific quantities related to the empirical measures at branching times, namely $\gamma^N_{\tau_j}(Q)$, $\gamma^N_{\tau_j^-}(Q)$, $\gamma^N_{\tau_j}(Q^2)$, and $\gamma^N_{\tau_j^-}(Q^2)$. \medskip

In fact, we will also need the following minor variant of $\gamma^N_{\tau_j^-}$, denoted $\gamma^{-,N}_{\tau_j}$ and defined by
$$
\gamma^{-,N}_{\tau_j} \eqdef \rho^N_{\tau_j^-}\ \frac1N  \sum_{n \in \Alive_j} \delta_{X^n_{\tau_j^-}}
= \p{1-\k_N/N}^{j-1}\ \frac1N  \sum_{n \in \Alive_j} \delta_{X^n_{\tau_j^-}}
$$

\begin{Lem}\label{lem:cvvar}
For $l=1,2$, we have the following convergences:
\begin{align*}
\gamma^N_{\tau_j}(Q^l) & \xrightarrow[N \to + \infty]{\P} \gamma_{t_j}(Q^l), \\
\gamma^N_{\tau_j^-}(Q^l) & \xrightarrow[N \to + \infty]{\P} \gamma_{t_j}(Q^l), \\
\gamma^{-,N}_{\tau_j} (Q^l) & \xrightarrow[N \to + \infty]{\P} \gamma_{t_j}(Q^l).
\end{align*}
\end{Lem}

\begin{proof}
We start by noting that $\gamma^N_{\tau_j^-}$ and $\gamma^{-,N}_{\tau_j}$ only differ by one Dirac measure of mass $1/N$, corresponding to the particle killed exactly at time $\tau_j$. Therefore, 
$$\gamma^N_{\tau_j^-}(Q^l)-\gamma^{-,N}_{\tau_j} (Q^l) = O(1/N),$$
and the second convergence will imply the third.\medskip 

Now we consider the first convergence, with $l=2$. 
Let $\varepsilon>0$. By Lemma \ref{lem:cont}, we can find $\delta>0$ such that $|\gamma_{t_j+\delta}(Q^2)-\gamma_{t_j-\delta}(Q^2)|\leq\varepsilon$. We consider that the event ${\cal A}_j^\delta=\{t_j-\delta\leq \tau_j\leq t_j+\delta\}$ is realised. By Proposition~\ref{pro:quant}, this happens with arbitrarily large probability for $N$ large enough.  By Lemma \ref{lemgQ2}, we have
$$
\gamma^N_{\tau_j}(Q^2)-\gamma^N_{t_j-\delta}(Q^2)=\int_{t_j-\delta}^{\tau_j} \rho^N_{t^-} d\A_t + \frac{1}{\sqrt N} \int_{t_j-\delta}^{\tau_j}  \rho^N_{t^-} d\widetilde{\M}_t.
$$
Inequality (\ref{lemgQ24}) implies that the second term tends to $0$ in probability. For the first one, using that $\A$ is increasing and again Lemma  \ref{lemgQ2}, we get
$$\int_{t_j-\delta}^{\tau_j} \rho^N_{t^-} d\A_t \leq\int_{t_j-\delta}^{t_j+\delta} \rho^N_{t^-} d\A_t =\gamma^N_{t_j+\delta}(Q^2)-\gamma^N_{t_j-\delta}(Q^2) + O_P(1/\sqrt{N})$$
with, by Proposition \ref{pro:estimate},
$$\left|\gamma^N_{t_j+\delta}(Q^2)-\gamma^N_{t_j-\delta}(Q^2) + O_P(1/\sqrt{N})\right|\xrightarrow[N\to\infty]{\P} \left|\gamma_{t_j+\delta}(Q^2)-\gamma_{t_j-\delta}(Q^2)\right| \leq 2\varepsilon.$$
We have then shown that $\gamma^N_{\tau_j}(Q^2)\to\gamma_{t_j}(Q^2)$ in probability. The second convergence for $l=2$ is proved the same way, with $\tau_j^-$ instead of $\tau_j$.\medskip 

We consider now $l=1$. Recall that $\gamma_{t_j+\delta}(Q)=\gamma_{t_j-\delta}(Q)$.
%On $A_k^\delta$, $\gamma^N_{\tau_j}$ and $\eta^N_{\tau_j}$ are equal up to a multiplicative deterministic constant $(1-K/N)^{k-1}\rightarrow \theta^{k-1}$, so we have also that $\eta^N_{\tau_j}(Q^2)\to\eta_{t_k}(Q^2)$. 
%
%%%% A METTRE PLUS LOIN
%And from the definition of $\bX$, we have 
%$$
%\frac{1}{N-K}\sum_{\Alive_k}\bQ(\bX_{\tau_j+k-1}^m)^2=\frac{N}{N-K}\eta_{\tau_j^-}^N(Q^2)+0_P(1/N)\rightarrow\eta_{t_k}(Q^2).
%$$
%%%%%%%%%%%%%%%
Hence, using the same kind of arguments, we obtain, for $N$ large enough so that $\P({\cal A}_j^\delta)>1-\varepsilon$,
\begin{align*}
&\E\left[|\gamma^N_{\tau_j}(Q)-\gamma^N_{t_j-\delta}(Q)|^2 \right]\\
&=4\|\varphi\|_\infty^2\varepsilon+\E\left[\un_{{\cal A}_j^\delta}\frac{1}{N}\int_{t_j-\delta}^{\tau_j} (\rho_{s^-}^N)^2(d[\M,\M]_s+d[\calM,\calM]_s)\right]\\
&\leq 4\|\varphi\|_\infty^2\varepsilon+\E\left[\int_{t_j-\delta}^{t_j+\delta} (\rho_{s^-}^N)^2(d[\M,\M]_s+d[\calM,\calM]_s)\right]\\
&= 4\|\varphi\|_\infty^2\varepsilon+\E\left[\frac{1}{N} (\gamma_{t_j+\delta}^N(Q)-\gamma_{t_j-\delta}^N(Q))^2\right]\\
&\leq 4\|\varphi\|_\infty^2\varepsilon+2\E\left[(\gamma^N_{t_j+\delta}(Q)-\gamma_{t_j+\delta}(Q))^2+(\gamma^N_{t_j-\delta}(Q)-\gamma_{t_j-\delta}(Q))^2\right]\\
&\leq4\|\varphi\|_\infty^2\varepsilon+\frac{16\|\varphi\|_\infty^2}{N},
\end{align*}
the last inequality coming from Proposition \ref{pro:estimate}. The latter implies the convergence in probability. The second convergence is again treated similarly, with $\tau_j^-$ instead of $\tau_j$.

\end{proof}
\subsection{Stretching the time}\label{sec:stretch}
The martingale $\gamma^N_t(Q^{T-t}(\varphi))$ has a quadratic variation with both continuous time and discrete time features. In order to show a CLT for its final value $\gamma^N_T(\varphi)$, we have to apply a general CLT for martingales with jumps. The problem is that the jumps at the resampling times do not get smaller when $N\to\infty$. To circumvent this difficulty, we first show a CLT  specifically tailored for our purpose.
 
\medskip

\begin{Pro}\label{pro:clt} Let $T > 0$ denote a fixed time horizon. For each $N \geq 1$, we consider on a filtered probability space the following random objects: first, a sequence of increasing stopping times $\tau_j$, $1 \leq 1 \leq \jmax$ with the convention $\tau_0=t_0=0$; and, second, a \cadlag local martingale $ t \mapsto M_t$
% Let $((M_t^N)_{0\leq t\leq T})_{N>0}$ denote a sequence of continuous time \cadlag bounded local martingales for the filtrations $\{\calF^N_t, 0\leq t\leq T\}$. Assume that, for each $N$, we have a  
% We suppose that $M^N$ 
that can be decomposed as the sum of two \cadlag local martingales, namely  $M_t=M_0+M_t^{0}+M_t^{1}$, where $M^0_0=M^1_0=0$, and $M^1$ is a pure jump martingale which jumps only at the stopping times $\tau_1,\dots,\tau_\jmax$ with jumps of the form 
$$
\Delta M^1_{\tau_j} = \sum_{m=1}^{K_N} \Delta_j^m \qquad j=1, \ldots \jmax,
$$
where $K_N$ is deterministic and $(\Delta_j^m)_{1\leq m\leq K_N}$ are integrable martingale increments with respect to a discrete filtration $\p{ \calF_j^m } _{0\leq m \leq K_N}$ verifying
$$
\calF_{\tau_j^-} \subset \calF_j^0 \subset \calF_j^1 \subset \dots \subset \calF_j^{K_N}=\calF_{\tau_j};
$$
that is, for $1 \leq m\leq K_N$, $\Delta_j^m$ is $\calF_j^m$-mesurable, and $\E\b{\Delta_j^m | \calF_j^{m-1}} = 0$. We also assume that $M^0_{\tau_j}$ is $\calF_j^0$ mesurable, making $M^0$ and $M^1$ orthogonal local martingales. \medskip

We then assume that all these objects satisfy the following properties:
\begin{enumerate}
\item \label{item1} $M_0$ converges in distribution towards $\mu_0$, a probability measure on $\R$.
\item \label{item2}For $1\leq j\leq\jmax$, $\tau_j \xrightarrow[N\to\infty]{\P} t_j$, for some deterministic sequence $0<t_1<\dots<t_j<\dots<t_\jmax<T$. 
\item  \label{item3} There exists a \cadlag increasing process $(v_t^N)_{0\leq t\leq T}$ such that $((M_t^{0})^2-v^N_t)_{0\leq t\leq T}$ is a local martingale. There is a deterministic continuous increasing function $v(t)$,  $0\leq t\leq T$, such that $v(0)=0$, and for  $0\leq t\leq T$, 
$$
 v^N_t \xrightarrow[N\to\infty]{\P} v(t).
$$
\item \label{item4} We have
$$
\lim_{N\to\infty} \E\b{\sup_{0\leq t\leq T} \left|M^{0}_t - M^{0}_{t^-}\right|^2} = 0,
$$
and
$$
\lim_{N\to\infty} \E\b{\sup_{0\leq t\leq T} \left|v^{N}_t - v^{N}_{t^-}\right|} = 0.
$$
\item \label{item5}
%At each $\tau^N_j$, $0<j\leq\jmax$, $M^{1}$ has a jump that can be decomposed as a sum of $K_N$ ``little'' jumps $\Delta_j^m$, $1\leq m \leq K_N$, that are martingale increments for their natural filtration, augmented by $\calF^N_{\tau^{N,-}_j}$\arnaud{préciser la filtration et reformuler}. We also assume that 
The sequence $(K_N)$ goes to infinity and, for each $1\leq j\leq\jmax$, we have
$$
\lim_{N\to\infty} \E\b{\max_{1\leq m\leq K_N}  | \Delta_j^m  |^2} = 0.
$$
\item \label{item6}For each $1\leq j\leq\jmax$, there is a deterministic continuous increasing function $\alpha \mapsto v_j(\alpha)$ on $[0,1]$, such that $v_j(0)=0$ and for all $\alpha\in[0,1]$,
$$
 \sum_{m=1}^{\lfloor \alpha K_N\rfloor} | \Delta_j^m  |^2   \xrightarrow[N\to\infty]{\P} v_j(\alpha).
$$
\end{enumerate}
Then, when $N \rightarrow +\infty$, the couple $(M_0,M_T-M_0)$ converges in distribution towards the tensor product between $\mu_0$ and a centered Gaussian variable with variance
$$
\sigma_T^2 \eqdef v(T)+\sum_{j=1}^\jmax v_j(1).
$$
\end{Pro}
\begin{proof}
For simplicity, we will consider the case where $M_0=0$. The general case can obtained by the same reasoning as in the proof of Theorem~$3.22$ of~\cite{cdgr2}. \medskip

We first construct a new local martingale $\bM$, which coincides with $M$ at the terminal time $T$, and which fulfills the assumptions of Theorem $1.4$ page~$339$ in~\cite{ek86}. The idea is to keep the same martingale between the $\jmax$ stopping times $(\tau_1, \ldots, \tau_{\jmax})$, and to ``stretch'' the time at each of the latter by inserting a time interval of length $1$. Each of the additional stretched time interval of length $1$ is then divided into exactly $K_N$ sub-intervals of length $\frac{1}{K_N}$; on the latter, the new, extended martingale, is piecewise constant and performs jumps with amplitudes $\Delta_j^m$ at the times
$$
\tau_j + j -1 + \frac{m}{K_N}, \quad (j,m) \in \set{1, \ldots, \jmax} \times \set{1, \ldots, K_N}.
$$
In the present proof (and only here), we will use the convention $$\tau_{\jmax+1}=t_{\jmax+1}=T.$$

We can now define the new martingale as
$$
\bM_s=\sum_{j=1}^{\jmax+1} \int_{(\tau_{j-1}+j-1)\wedge s}^{(\tau_{j}+j-1)\wedge s} dM^{0}_{s-j+1} + \sum_{j=1}^\jmax \sum_{m=1}^{K_N} \Delta_j^m \un_{t\geq \tau_j+j-1+\frac{m}{K_N}},
$$
where $0\leq s \leq T+\jmax$ denotes the new time index for the time stretched processes. Formally, we introduce the \cadlag integer-valued processes
$$
s \mapsto j^N_s \eqdef \inf \set{ j \geq 0,\ \tau_{j+1} + j > s}
$$
which counts the number of stopping times $(\tau_j)_{j \geq 1}$ that are encountered before $s$ on the stretched time interval. Then two cases are possible. Case~$(i)$: $s$ belongs to an inserted stretching time interval, that is
$$
\tau_{j^N_s}+j^N_s-1+\frac{m^N_s-1}{K_N} \leq s < \tau_{j^N_s}+j^N_s-1+\frac{m^N_s}{K_N} 
$$
for some $1 \leq m^N_s \leq K_N$ which defines which of the $K_N$ sub-intervals $s$ belongs to. We can then naturally define the original (non-stretched) time as
$$
t^N_s \eqdef \tau_{j^N_s}.
$$
Case~$(ii)$:  $s$ does not belong to an inserted (i.e., due to stretching) time interval, that is
$$
\tau_{j^N_s}+j^N_s \leq s < \tau_{j^N_s+1}+j^N_s,
$$
in which case we naturally set $m^N_s = K_N $ as well as 
$$
t^N_s := s - j^N_s.
$$
In any case, we are led to
$$
t^N_s = (s - j^N_s) \wedge \tau_{j^N_s}.% = \max_{j \geq 0 }  (s - j) \vee \tau_{j} .
$$
The latter obviously defines two \cadlag processes $s \mapsto m^N_s$ and $s \mapsto t^N_s$. Note that $t^N_s$ is a $(\calF_t)_{t \geq 0}$ stopping time for each $s \geq 0$. \medskip
%s \mapsto J_s \eqdef \argmin
%$$
%$$
%A \in \bcF_s \Leftrightarrow 
%\begin{cases}
% A \cap \set{J_s =j\, \& K_s = k}
%\end{cases}

We also need do define the extended filtration naturally associated with this new time, and with respect to which ~$\bM$ is indeed a martingale, and such that the processes $s \mapsto (j^N_s,m^N_s,t^N_s)$ are adapted. This can be done by setting
$$A \in \bcF_s \Leftrightarrow 
 A \cap \set{j^N_s \leq j,m^N_s \leq m,t^N_s \leq t} \in \calF_{t} \vee \calF^m_j \qquad \forall j\geq 0, m \geq 1, t \geq 0$$
%\begin{align*}
%& A \in \bcF_s \Leftrightarrow \\
%& \qquad 
%\begin{cases}
% &A \cap \set{j^N_s \leq j \, \& \, m^N_s \leq m \, \& \, t^N_s \leq t} \in \calF_{t} \bigvee \calF^m_j \, \, \forall j\geq 0, m \geq 0, t \geq 0.
%\end{cases}
%\end{align*}
or, equivalently, 
$$
\bcF_s=\left[ \bigvee_{j=1}^{\jmax+1} \calF_{(s-j+1)\wedge\tau_j}   \right] \vee \sigma(A \cap \{j^N_s \leq j,m^N_s \leq m\}, j \geq 0, m \geq 1, A \in \calF^m_j )
$$
%
%$$
%\left[ \bigvee_{j=1}^{\jmax+1} \calF_{(s-j+1)\wedge\tau_j}  \vee \sigma(M^{0}_{\tau_j}\un_{s-j+1\geq \tau_j})\right] \vee \left[ \bigvee_{j=1}^\jmax  \bigvee_{m=1}^{K_N} \sigma(\Delta_j^m \un_{s\geq \tau_j+j-1+\frac{m}{K_N}}) \right]
%$$
%
% $$
% \bcF_s=\left[ \bigvee_{j=1}^{\jmax+1} \calF_{(s-j+1)\wedge\tau_j}  \vee \sigma(M^{0}_{\tau_j}\un_{s-j+1\geq \tau_j})\right] \vee \left[ \bigvee_{j=1}^\jmax  \bigvee_{m=1}^{K_N} \sigma(\Delta_j^m \un_{s\geq \tau_j+j-1+\frac{m}{K_N}}) \right]
% $$
so that, by Doob's optional sampling theorem, $\bM$ is an $\bcF$--martingale. We also remark that, on the event $\{ \tau_\jmax<T\}$, we have $\bM_{T+\jmax}=M_T$.\medskip 

For $0\leq s \leq T+\jmax$, we next define the large $N$ limit of the processes $(j_s^N)$, $(t_s^N)$ and $\frac{m^N_s}{K_N}$, which are respectively
$$
j_s \eqdef  \inf \set{ j \geq 0,\ t_{j+1} + j > s} = \sum_{j=1}^\jmax \un_{s\geq t_j+j-1},
$$
$$
t_s \eqdef (s-j_s) \wedge t_{j_s},
$$
and
$$
m_s\eqdef (s-j_{s})\wedge 1,
$$
as well as the asymptotic variance by 
\begin{align*}
c(s)= v\p{t_s} + \sum_{j=1}^{j_s-1} v_{j}(1) + v_{j_{s}}(m_s) .
\end{align*}
It is easily checked that the limit $c(s)$ is continuous. We finally define a quadratic variation $\bv^N$ for $\bM$ by
$$
\bv^N_s= v^N_{t^N_s}+ \sum_{j=1}^{j^N_s-1} \sum_{m=1}^{K_N} (\Delta^m_j)^2 + \sum_{m=1}^{m^N_s}( \Delta_{j^N_s}^m)^2.
$$
% \begin{align*}
% \bv^N_s=&a^N_{s-\sum_{j=1}^\jmax \un_{s\geq \tau_j+j} -\sum_{j=1}^\jmax \un_{ \tau_j+j-1\leq s< \tau_j+j}(s-\tau_j+j-1)}\\
% &+ \sum_{j=1}^\jmax \un_{\tau_j+j\leq s}  \sum_{m=1}^{K_N} (\Delta^m_j)^2 \\
% &+ \sum_{j=1}^\jmax \un_{\tau_j+j-1\leq s< \tau_j+j}\sum_{m=1}^{\lfloor(s-\tau_j+1)K_N\rfloor}( \Delta_j^m)^2.
% \end{align*}
It is clear that $(\bM_s)^2-\bv^N_s$ is a local martingale. \medskip 

From items \ref{item1},  \ref{item2},  \ref{item3},  \ref{item6}, we can check that for all $0\leq s \leq T+\jmax$, $\bv^N_s$ goes to  $c(s)$ in probability. More precisely, because the processes are increasing, and $t^N_s\to t_s$ in probability, for any $\delta>0$, for $N$ large enough, we have with arbitrarily large probability that 
$$
v^N_{t_s-\delta}\leq v^N_{t^N_s} \leq v^N_{t_s+\delta},
$$
with $v^N_{t_s-\delta}\to v_{t_s-\delta}$ and $v^N_{t_s+\delta}\to v_{t_s+\delta}$. By taking $\delta$ small enough, we can have $v_{t_s+\delta}-v_{t_s-\delta}$ arbitrarily small by continuity of the limit, which proves the convergence for the first term. The third term can be treated similarly, and is not detailed. 

Moreover, the assumptions on the jumps in Theorem $1.4$ page 339 in~\cite{ek86} are verified by items \ref{item4} and \ref{item5}. Therefore the process $\bM$ converges in distribution to a Gaussian process with variance given by $c(s)$. In particular, we have the convergence in distribution of the final time marginal $\bM_{T+\jmax}$. To finally transfer the convergence to $M_T$, we write
$$
|M_T-\bM_{T+\jmax}|=\un_{\tau_\jmax>T} |M_T-\bM_{T+\jmax}|\leq \un_{\tau_\jmax>T} \ 2\|M\|_\infty,
$$
which converges in probability to $0$ from item \ref{item2}. 
\end{proof}

\medskip

\subsection{Proof of the main result}
In this section we use Proposition \ref{pro:clt} to prove our main result Theorem \ref{gamma}. Recall that $1-K/N$ goes to $\theta$ when $N$ goes to infinity.

\begin{Lem}
\label{lem:delta}
At each resampling time $\tau_j$, the local martingale
$$
%\frac{1}{\sqrt{N}}
\int_0^t \rho^N_{u^-} d\calM_u,
$$
 jumps, and each jump can be decomposed as a sum of $K$ martingale increments $\Delta_j^m$, for $1\leq m\leq K$,
$$
 %\frac{1}{\sqrt{N}} 
 \rho^N_{\tau_j^-} \Delta \calM_{\tau_j} = \sum_{m=1}^K \Delta_j^m.
$$
Moreover, we have, for any sequence $\alpha_N \to \alpha\in (0,1]$, 
$$
\sum_{m=1}^{\lfloor \alpha_N K\rfloor} (\Delta_j^m)^2  \xrightarrow[N\to\infty]{\P} \alpha \theta^{2j}(1-\theta) \V_{\eta_{t_j}}(Q).
$$
\end{Lem}
\begin{proof}
We first detail the proof for the simpler case $\alpha_N=1$. 
We have (see (\ref{lajcaljcljsc}) in the proof of Lemma \ref{lem:decomp})
\begin{align*}
%\frac{1}{\sqrt{N}}
\rho_{\tau_j^-}^N \Delta \calM_{\tau_j}& = \frac{1}{\sqrt N} (1-\k/{N})^{j-1} \p{ \sum_{n \notin \Alive_j} \L^n_{\tau_j}   - \frac{\k}{N-\k} \sum_{n \in \Alive_j} \L^n_{\tau_j} }\\
&=  (1-\k/{N})^{j} \sum_{n \notin \Alive_j} \frac{1}{\sqrt N}\p{\L^n_{\tau_j}  - \frac{1}{N-\k} \sum_{n \in \Alive_j} \L^n_{\tau_j}}.
\end{align*}
It is easy to check that this last expression is a sum of $K$ martingale increments
$$\Delta_j^m=\frac{1}{\sqrt N} (1-\k/{N})^{j}\p{\L^m_{\tau_j}  - \frac{1}{N-\k} \sum_{n \in \Alive_j} \L^n_{\tau_j}}.$$
Given $\calF_{\tau_j}^-$, it is actually a sum of i.i.d. uniformly bounded variables, and so is the sum of their squares, so that it is concentrated around its mean (e.g., by Chebyshev's %Bernstein
 inequality)
\begin{align*}
\sum_{m=1}^K (\Delta_j^m)^2&=(1-\k/{N})^{2j-2} \p{(1-K/N)^2\frac{\k}{N}\V_{\eta^N_{\Alive_j}}  +O_P(1/\sqrt N)}.
\end{align*}
So, using Lemma \ref{lem:cvvar}, we finally get
$$
\sum_{m=1}^K (\Delta_j^m)^2\xrightarrow[N\to\infty]{\P} \theta^{2j}(1-\theta) \V_{\eta_{t_j}}(Q).
$$

For a general $\alpha_N$, the proof is similar, except that instead of all the particles in $\Alive_j$, we take only the $\lfloor \alpha_N K\rfloor$ first ones. Note that the chosen ordering of $\Alive_j$ is irrelevant  because the new particles are i.i.d. (given $\calF_{\tau_j}^-$).
\end{proof}
\begin{Lem}
\label{lem:var2}
For $0\leq j\leq \jmax $ and $0\leq t\leq T$, we have 
$$
\int_{\tau_j\wedge t}^{\tau_{j+1}\wedge t}  (\rho^N_{u^-})^2 d\A_u \xrightarrow[N\to\infty]{\P} \theta^j (\gamma_{t_{j+1}\wedge t}(Q^2) - \gamma_{t_j\wedge t}(Q^2)).
$$
\end{Lem}
\begin{proof}
%If $t_j<t\leq t_{j+1}$, then 
%
The proof is quite similar to that of Lemma \ref{lem:cvvar}. For $\tau_{j} < u\leq \tau_{j+1}$, we have $\rho^N_{u^-}=(1-K/N)^j$, so that
$$
\int_{\tau_j\wedge t}^{\tau_{j+1}\wedge t}  (\rho^N_{u^-})^2 d\A_u =(1-K/N)^j \int_{\tau_j\wedge t}^{\tau_{j+1}\wedge t}  \rho^N_{u^-} d\A_u.% + \un_{\tau_j\leq t}(1-1/N)^{2j-2} \Delta \A_{\tau_i}. 
$$ 
%From Lemma \ref{ajump} $\Delta \A_{\tau_i}$ is of order $1/N$ and goes to $0$. 
By assumption, the deterministic factor $(1-K/N)^j$ goes to $\theta^j$ when $N$ goes to infinity. Additionally, since $\A$ is increasing and $\rho^N_{u^-} >0$, we deduce that, for any $\delta>0$, 
$$
\int_{(\tau_j\wedge t)+\delta}^{(\tau_{j+1}\wedge t)-\delta}  \rho^N_{u^-} d\A_u  \leq \int_{\tau_j\wedge t}^{\tau_{j+1}\wedge t}  \rho^N_{u^-} d\A_u \leq \int_{(\tau_j\wedge t)-\delta}^{(\tau_{j+1}\wedge t)+\delta}  \rho^N_{u^-} d\A_u.
$$
From Lemma \ref{lemgQ2}, we have
\begin{align*}
&\gamma^N_{(\tau_{j+1}\wedge t)-\delta}(Q^2)-\gamma^N_{(\tau_j\wedge t)+\delta}(Q^2) +O_P(1/\sqrt N)\\
 &\ \ \leq \int_{\tau_j\wedge t}^{\tau_{j+1}\wedge t}  \rho^N_{u^-} d\A_u\leq \gamma^N_{(\tau_{j+1}\wedge t)+\delta}(Q^2)-\gamma^N_{(\tau_j\wedge t)-\delta}(Q^2) +O_P(1/\sqrt N).
\end{align*}
From Proposition~\ref{pro:quant}, for $N
$ large enough, with large probability, it then comes
\begin{align*}
&\gamma^N_{(t_{j+1}\wedge t)-\delta}(Q^2)-\gamma^N_{(t_j\wedge t)+\delta}(Q^2) +O_P(1/\sqrt N)\\
&\ \ \leq \int_{\tau_j\wedge t}^{\tau_{j+1}\wedge t}  \rho^N_{u^-} d\A_u
\leq \gamma^N_{(t_{j+1}\wedge t)+\delta}(Q^2)-\gamma^N_{(t_j\wedge t)-\delta}(Q^2) +O_P(1/\sqrt N).
\end{align*}
Proposition \ref{pro:estimate} implies 
$$
\gamma^N_{(t_{j+1}\wedge t)-\delta}(Q^2)-\gamma^N_{(t_j\wedge t)+\delta}(Q^2)\xrightarrow[N\to\infty]{\P} \gamma_{(t_{j+1}\wedge t)-\delta}(Q^2)-\gamma_{(t_j\wedge t)+\delta}(Q^2),
$$
and 
$$
\gamma^N_{(t_{j+1}\wedge t)+\delta}(Q^2)-\gamma^N_{(t_j\wedge t)-\delta}(Q^2)\xrightarrow[N\to\infty]{\P} \gamma_{(t_{j+1}\wedge t)+\delta}(Q^2)-\gamma_{(t_j\wedge t)-\delta}(Q^2).
$$
By continuity of the mapping $t\mapsto\gamma_t(Q^2)$, see Lemma \ref{lem:cont}, we can choose $\delta$ small enough such that the difference of the two limits is arbitrarily small, both being close to $ \gamma_{t_{j+1}\wedge t}(Q^2)-\gamma_{t_j\wedge t}(Q^2)$. 
\end{proof}

\paragraph{Proof of Theorem \ref{gamma}}
Recall that
$$\gamma_T^N(\ph)-\gamma_T(\ph)=\Big(\gamma_T^N(Q)-\gamma^N_0(Q)\Big)
+\Big(\eta^N_0(Q^T(\ph))-\eta_0(Q^T(\ph))\Big),$$
where, by  \eqref{eq:decomp},
$$\sqrt{N}(\gamma_t^N(Q)-\gamma_0^N(Q))= \int_0^t \rho^N_{u^-} \p{ d\M_u + d\calM_u}.$$
It turns out that the martingale $\sqrt{N}\gamma_t^N(Q)$ does not satisfy the assumptions of Proposition \ref{pro:clt} because the number of resamplings is not a priori bounded. We therefore define a new martingale by setting the initial condition $M_0 \eqdef \sqrt{N}(\eta^N_0(Q^T(\ph))-\eta_0(Q^T(\ph))$ as well as
$$M_t - M_0=\sum_{j=1}^\jmax \un_{\tau_j\leq t}\rho^N_{\tau_j^-}\Delta \calM_{\tau_j} +  \sum_{j=1}^\jmax  \int_{\tau_{j-1}\wedge t}^{\tau_j\wedge t} \rho_{u^-}^N d\M_u + \int_{\tau_\jmax \wedge t}^{T\wedge t\wedge \tau_{\jmax+1}} \rho_{\tau_\jmax}^N d\M_u.
$$
Simple algebra reveals that 
$$  
\int_0^T \rho_{u^-}^N (d\M_u+d\calM_u) - (M_T - M_0) \neq 0
$$
implies that $\tau_{\jmax+1}\leq T$. But by Proposition~\ref{pro:quant}, this happens with arbitrarily small probability, provided $N$ is large enough, so that 
$$
\left|\int_0^T \rho_{u^-}^N (d\M_u+d\calM_u) - (M_T-M_0)\right|\xrightarrow[N\to\infty]{\P} 0.
$$
As a consequence, it suffices to show the CLT for  $M_t=M_0+M^{0}_t+M^{1}_t$.

The rest of the proof is devoted to show that $M_t$ indeed satisfies the assumptions of Proposition \ref{pro:clt}, with 
$$
 M^{1}_t = \sum_{j=1}^\jmax \un_{\tau_j\leq t}\rho^N_{\tau_j^-}\Delta \calM_{\tau_j},
$$
and
$$
M^{0}_t=\sum_{j=1}^\jmax  \int_{\tau_{j-1}\wedge t}^{\tau_j\wedge t} \rho_{u^-}^N d\M_u  + \int_{\tau_\jmax \wedge t}^{T\wedge t\wedge \tau_{\jmax+1}} \rho_\jmax^N d\M_u.
$$
For $M^{1}_t$, item \ref{item2} comes from Proposition~\ref{pro:quant}, item \ref{item5} is from the construction of the particle system and the fact that each little jump is of order $1/\sqrt{N}$, and item \ref{item6} is from Lemma \ref{lem:delta}.\medskip 

For $M^{0}_t$, we define the increasing process $v^N_t$ as 
$$
v^N_t=\sum_{j=1}^\jmax  \int_{\tau_{j-1}\wedge t}^{\tau_j\wedge t} (\rho_{u^-}^N)^2 d\A_u  + \int_{\tau_\jmax \wedge t}^{T\wedge t\wedge \tau_{\jmax+1}} (\rho_\jmax^N)^2 d\A_u.
$$
The fact that $(M^{0}_t)^2-v^N_t$ is a local martingale is from Lemma \ref{Lem:quad}. Item \ref{item3} is from Lemma \ref{lem:var2}, item \ref{item4} from Lemma \ref{lem:jumps} (ii) and (iii) (we should not forget the $1/\sqrt N$ factor in the definition of $\M$ in equation (\ref{mbbt})), and  Lemma \ref{Lem:quad}. The orthogonality between $M^{0}$ and $M^{1}
$ is from Lemma \ref{Lem:quad}. \medskip

The convergence of the initial condition $M_0 \eqdef \sqrt{N}(\eta^N_0(Q^T(\ph))-\eta_0(Q^T(\ph))$ is the usual CLT. \medskip

We can then apply Proposition \ref{pro:clt}. Lemmas \ref{lem:delta} and \ref{lem:var2} imply that the total asymptotic variance of $\sqrt{N} \p{ \gamma_T^N(\ph) - \gamma_T(\ph)}$ is given by 

\begin{align*}
\sigma^2_T(\ph)&= \Var_{\eta_0}(Q) + \sum_{j=1}^\jmax \theta^{2j}(1-\theta) \V_{\eta_{t_j}}(Q)\nonumber \\
& \qquad + \sum_{j=0}^{\jmax-1} \theta^{j}(\gamma_{t_{j+1}}(Q^2)-\gamma_{t_j}(Q^2))+ \theta^{\jmax}(\gamma_{T}(Q^2)-\gamma_{t_{\jmax}}(Q^2)) . %\label{eq:var_crude}
\end{align*}

We recall that by definition $\eta_{t} = \gamma_{t} / \rho_t$ and that $\rho_{t} = \theta^j $ for $\theta^{j} \leq t < \theta^{j+1}$, so that we can rewrite the asymptotic variance as
\begin{align*}
\sigma^2_T(\ph)&=\eta_0(Q^2)- \eta_0(Q)^2 +\sum_{j=1}^\jmax \theta^{2j}(1-\theta)\p{\eta_{t_j}(Q^2)-\eta_{t_j}(Q)^2}\\
& \qquad + \sum_{j=0}^{\jmax-1} \theta^{2j}(\theta\eta_{t_{j+1}}(Q^2)-\eta_{t_j}(Q^2))+ \theta^{2\jmax}(\eta_T(Q^2)-\eta_{t_\jmax}(Q^2)).
\end{align*}
%so that, using 
%$$\sum_{j=1}^\jmax \theta^{2j}\eta_{t_j}(Q^2)- \sum_{j=0}^{\jmax-1} \theta^{2j} \eta_{t_{j}}(Q^2) =  \theta^{2\jmax} \eta_{t_\jmax}(Q^2)-\eta_0(Q^2),$$
This may be reformulated as
\begin{align}
\sigma^2_T(\ph)&=\eta_0(Q^2)- \eta_0(Q)^2 +\sum_{j=0}^{\jmax-1} \theta^{2j+1}\eta_{t_{j+1}}(Q^2)+ \theta^{2\jmax}\eta_T(Q^2)- \eta_{0}(Q^2) \nonumber \\
& \qquad -\sum_{j=1}^\jmax \theta^{2j+1}\eta_{t_j}(Q^2) - \sum_{j=1}^\jmax (1-\theta)\theta^{2j}\eta_{t_j}(Q)^2 , \nonumber \\
&= \theta^{2\jmax}\V_{\eta_{T}}(Q) + \sum_{j=1}^{\jmax} \p{\theta^{2j-1}  - \theta^{2j+1}  }\eta_{t_{j}}(Q^2) - \sum_{j=1}^\jmax \theta^{2j}(1-\theta)\eta_{t_j}(Q)^2 , \label{eq:var_simple}
\end{align}
where, in the last line, we have used that by definition $\eta_T(Q^2) = \eta_T(\ph^2) $ and $\eta_0(Q) = \gamma_T(\ph) = \theta^{\jmax} \eta_T(\ph)$.
Finally, remarking that 
$$   \theta^{2j}(1-\theta) = (\theta^{2j-1}  - \theta^{2j+1})-\theta^{2j} ( 1/ \theta -1),$$
as well as  
$$\eta_{t_j}(Q) = \gamma_{t_j}(Q) \theta^{-j} = \gamma_T(\ph) \theta^{-j} = \eta_T(\ph) \theta^{\jmax-j},$$ 
we conclude that
$$
\sigma^2_T(\ph) = \theta^{2\jmax}(\V_{\eta_{T}}(\ph)+ \jmax (1/\theta-1)\eta_{T}(\ph)^2)+\sum_{j=1}^{\jmax} \p{\theta^{2j-1}  - \theta^{2j+1}  }\V_{\eta_{t_{j}}}(Q).
$$

Hence we have proved Theorem \ref{gamma} for any test function $\ph$ in ${\cal D}$. To see that the result is still valid for any $\ph$ in $\overline{\cal D}$, it suffices to apply the same reasoning as in \cite{cdgr2}.

\section{Supplementary material}
\label{akjajdnzjd}

\subsection{Another formulation of the asymptotic variance}

As mentioned in  \cite{cdgr1}, it turns out that it is possible to make a connection between Fleming-Viot particle systems and interacting particle systems as exposed for example in the pair of books \cite{delmoral04a,del2013mean}. Without going into details, we will just show that our asymptotic variance coincides with the one given in \cite{delmoral04a} page 452. As already noticed in \cite{cdgr1}, we need to use predicted measures instead of corrected ones. At each $t_k$, we denote by $\tilde\eta_{t_k}$ the predicted measure, that is $\eta_{t_{k-1}}Q^{t_{k}-t_{k-1}}$. We have $\tilde\eta_{t_k}=\theta\eta_{t_k}+(1-\theta) \delta_\partial$. For any test function $\varphi$ such that $\varphi(\partial)=0$, we have $\tilde\eta_{t_k}(\varphi)=\theta\eta_{t_k}(\varphi)$. Note also that $\eta_T=\tilde\eta_T$ since there is no resampling at the end.\medskip

We start from \eqref{eq:var_simple} and remark that
$$
\sum_{j=1}^\jmax \theta^{2j}\eta_{t_j}(Q)^2 = \sum_{j=1}^\jmax \theta^{2(j-1)}\tilde\eta_{t_j}(Q)^2 %=\sum_{j=0}^{\jmax-1} \theta^{2j}\tilde\eta_{t_{j+1}}(Q)^2,
$$
to get, with the convention $t_{\jmax+1}=T$,
\begin{align*}
\sigma^2_T(\ph)&=\theta^{2\jmax}\V_{\tilde\eta_{T}}(Q)+\sum_{j=1}^{\jmax} \theta^{2j-2}\V_{\tilde\eta_{t_{j}}}(Q)-\sum_{j=1}^\jmax \theta^{2j}\theta\eta_{t_j}(Q^2)
+\sum_{j=1}^\jmax  \theta^{2j}\theta\eta_{t_j}(Q)^2\\
&=\sum_{j=0}^{\jmax} \theta^{2j}\V_{\tilde\eta_{t_{j+1}}}(Q)-\sum_{j=1}^\jmax \theta^{2j+1}\V_{\eta_{t_j}}(Q). 
%+\sum_{j=1}^\jmax  \theta^{2j}\theta\eta_{t_j}(Q)^2\\
%&=\sum_{j=0}^{\jmax} \theta^{2j}\V_{\tilde\eta_{t_{j+1}}}(Q)-\sum_{j=1}^\jmax \theta^{2j}\tilde\eta_{t_j}(Q^2)
%+\sum_{j=1}^\jmax  \theta^{2j}\theta^{-1}\tilde\eta_{t_j}(Q)^2.
\end{align*}
Now, we observe that
$$ \theta \V_{\eta_{t_j}}(Q) 
 = \tilde\eta_{t_j}(Q^2)-\theta\tilde\eta_{t_{j+1}}(Q)^2 =  \tilde\eta_{t_j}(\un_F (Q-\tilde\eta_{t_{j+1}}(Q))^2),
$$
so that
\begin{align*}
\sigma^2_T(\ph)&=\sum_{j=0}^{\jmax} \theta^{2j}\V_{\tilde\eta_{t_{j+1}}}(Q) - \sum_{j=1}^\jmax  \theta^{2j} \tilde\eta_{t_j}(\un_F (Q-\tilde\eta_{t_{j+1}}(Q))^2)\\
&=\sum_{j=1}^{\jmax+1} \theta^{2(j-1)}\V_{\tilde\eta_{t_{j}}}(Q) - \sum_{j=2}^{\jmax+1}  \theta^{2(j-1)} \tilde\eta_{t_{j-1}}(\un_F (Q-\tilde\eta_{t_{j}}(Q))^2).
\end{align*}
This is the result given in  \cite{delmoral04a} page 452.

\subsection{Stopping times and martingales}
\begin{Lem}\label{albcios}
 Let $\tau$ be a stopping time on a filtered probability space, and $U$ an integrable and $\calF_\tau$ measurable random variable such that $\E \b {U | \calF_{\tau^-}}=0$. Then the process $t \mapsto U \un_{t \geq  \tau}$ is a c\`adl\`ag martingale. 
\end{Lem}
\begin{proof}
 Let $t > s$ be given. First remark that $\un_{t \geq  \tau} = \un_{s \geq  \tau} +  \un_{s <  \tau} \un_{t \geq  \tau} $. Then by definition of $\calF_\tau$, $U\un_{s \geq  \tau} $ is $\calF_s$-measurable, so that
 \begin{align*}
 \E \b{ U \un_{t \geq  \tau} | \calF_s} =  U\un_{s \geq  \tau} + \E \b{U \un_{t \geq  \tau} 
  | \calF_s}\un_{s <  \tau}.
  \end{align*}
Next, by definition of $\calF_{\tau^-}$, $\E \b{U \un_{t \geq  \tau} 
  | \calF_s}\un_{s <  \tau} $ and $\un_{t \geq  \tau}$ are $\calF_{\tau^-}$-measurable, hence the result follows from 
\begin{align*}
   \E \b{U \un_{t \geq  \tau} 
  | \calF_s}\un_{s <  \tau}  =\E \b{\E \b{ U | \calF_{\tau^-} }\un_{t \geq  \tau} 
  | \calF_s}\un_{s <  \tau}  =0.
   \end{align*}
\end{proof}

\bibliographystyle{plain}
\bibliography{biblio-cdgr}
\end{document}

%% file: macros.tex
\newcommand{\tM}{\widetilde{\mathbb{M}}}
\newcommand{\Alive}{{\rm Alive}}
\newcommand{\tcalM}{\widetilde{\mathcal{M}}}

\renewcommand{\k}{{K}}

\newcommand{\bM}{\pmb{\mathbb{M}}}

\newcommand{\bv}{\pmb{{v}}}

\renewcommand{\bM}{\pmb{{M}}}

\newcommand{\bcF}{\pmb{\mathcal{F}}}

\newcommand{\un}{{\mathbf{1}}}

\newcommand{\V}{\mathbb V}

\definecolor{darkpastelgreen}{rgb}{0.01, 0.75, 0.24}

\renewcommand{\le}{\leqslant}
\renewcommand{\leq}{\leqslant}
\renewcommand{\ge}{\geqslant}
\renewcommand{\geq}{\geqslant}

\newcommand{\jmaxN}{{j^N_{\rm max}}}
\newcommand{\jmax}{{j_{\rm max}}}
\newcommand{\cadlag}{c\`adl\`ag }

\newcommand{\A}{\mathbb{A}}

\newcommand{\E}{\mathbb{E}}

\newcommand{\N}{\mathbb{N}}
\newcommand{\M}{\mathbb{M}}
\newcommand{\R}{\mathbb{R}}

\renewcommand{\P}{\mathbb{P}}
\renewcommand{\L}{\mathbb{L}}

\newcommand{\calM}{\mathcal{M}}
\newcommand{\calD}{\mathcal{D}}
\newcommand{\calF}{\mathcal{F}}
\newcommand{\calB}{\mathcal{B}}

\newcommand{\calN}{\mathcal{N}}

\newcommand{\eps}{\varepsilon}
\newcommand{\ph}{\varphi}

\newcommand{\Var}{\mathbb V}
\newcommand{\Card}{{\rm card}}

\newcommand{\eqdef}{:=}
\newcommand{\one}{{\mathbf{1}}}
\newcommand{\dps}{\displaystyle}

\newcommand{\set}[1]{\left\{#1\right\}}

\newcommand{\p}[1]{{\mathopen{}\left(#1\right)\mathclose{}}}
\renewcommand{\b}[1]{\left [ \, #1 \, \right ]}

\newcommand{\norm}[1]{\left\Vert#1\right\Vert}

\theoremstyle{plain}
\newtheorem{The}{Theorem}[section]
\newtheorem{Lem}[The]{Lemma}
\newtheorem{Pro}[The]{Proposition}

\newtheorem{Cor}[The]{Corollary}

\newtheorem{Def}[The]{Definition}
\newtheorem{Not}[The]{Notation}

\newtheorem{AssA}{Assumption}

\newtheorem{AssB}{Assumption}

\newtheorem{AssC}{Assumption}

\numberwithin{equation}{section}

\theoremstyle{definition}
\newtheorem{Rem}[The]{Remark}

%% file: fv1.pstex_t
\begin{picture}(0,0)%
\includegraphics{fv1.pstex}%
\end{picture}%
\setlength{\unitlength}{1973sp}%
\begingroup\makeatletter\ifx\SetFigFont\undefined%
\gdef\SetFigFont#1#2#3#4#5{%
  \reset@font\fontsize{#1}{#2pt}%
  \fontfamily{#3}\fontseries{#4}\fontshape{#5}%
  \selectfont}%
\fi\endgroup%
\begin{picture}(8655,6108)(511,-6832)
\put(526,-4561){\makebox(0,0)[lb]{\smash{{\SetFigFont{9}{10.8}{\rmdefault}{\mddefault}{\updefault}{\color[rgb]{0,0,0}$\eta_0={\cal L}(X_0)$}%
}}}}
\put(8326,-5536){\makebox(0,0)[lb]{\smash{{\SetFigFont{9}{10.8}{\rmdefault}{\mddefault}{\updefault}{\color[rgb]{0,0,0}$T$}%
}}}}
\put(9151,-5311){\makebox(0,0)[lb]{\smash{{\SetFigFont{9}{10.8}{\rmdefault}{\mddefault}{\updefault}{\color[rgb]{0,0,0}$t$}%
}}}}
\put(4126,-3661){\makebox(0,0)[lb]{\smash{{\SetFigFont{9}{10.8}{\rmdefault}{\mddefault}{\updefault}{\color[rgb]{0,0,0}$X_t$}%
}}}}
\put(6901,-5011){\makebox(0,0)[lb]{\smash{{\SetFigFont{9}{10.8}{\rmdefault}{\mddefault}{\updefault}{\color[rgb]{0,0,0}$\tau_\partial$}%
}}}}
\put(2026,-3361){\makebox(0,0)[lb]{\smash{{\SetFigFont{9}{10.8}{\rmdefault}{\mddefault}{\updefault}{\color[rgb]{0,0,0}$x_0$}%
}}}}
\put(6826,-6061){\makebox(0,0)[lb]{\smash{{\SetFigFont{9}{10.8}{\rmdefault}{\mddefault}{\updefault}{\color[rgb]{0,0,0}$X_t=\partial\ \forall t\geq\tau_\partial$}%
}}}}
\put(2476,-961){\makebox(0,0)[lb]{\smash{{\SetFigFont{9}{10.8}{\rmdefault}{\mddefault}{\updefault}{\color[rgb]{0,0,0}$F=(0,+\infty)$}%
}}}}
\put(4951,-1411){\makebox(0,0)[lb]{\smash{{\SetFigFont{9}{10.8}{\rmdefault}{\mddefault}{\updefault}{\color[rgb]{0,0,0}$p_T=\mathbb{P}(X_T\neq\tau_\partial)$}%
}}}}
\put(5326,-6736){\makebox(0,0)[lb]{\smash{{\SetFigFont{9}{10.8}{\rmdefault}{\mddefault}{\updefault}{\color[rgb]{0,0,0}$X_T=\partial$}%
}}}}
\end{picture}%

%% file: pt.pstex_t
\begin{picture}(0,0)%
\includegraphics{pt.pstex}%
\end{picture}%
\setlength{\unitlength}{1973sp}%
\begingroup\makeatletter\ifx\SetFigFont\undefined%
\gdef\SetFigFont#1#2#3#4#5{%
  \reset@font\fontsize{#1}{#2pt}%
  \fontfamily{#3}\fontseries{#4}\fontshape{#5}%
  \selectfont}%
\fi\endgroup%
\begin{picture}(7227,4899)(1939,-5623)
\put(5026,-2086){\makebox(0,0)[lb]{\smash{{\SetFigFont{9}{10.8}{\rmdefault}{\mddefault}{\updefault}{\color[rgb]{0,0,0}$p_t=\mathbb{P}(X_t\neq\tau_\partial)$}%
}}}}
\put(8326,-5536){\makebox(0,0)[lb]{\smash{{\SetFigFont{9}{10.8}{\rmdefault}{\mddefault}{\updefault}{\color[rgb]{0,0,0}$T$}%
}}}}
\put(9151,-5311){\makebox(0,0)[lb]{\smash{{\SetFigFont{9}{10.8}{\rmdefault}{\mddefault}{\updefault}{\color[rgb]{0,0,0}$t$}%
}}}}
\put(4276,-5461){\makebox(0,0)[lb]{\smash{{\SetFigFont{9}{10.8}{\rmdefault}{\mddefault}{\updefault}{\color[rgb]{0,0,0}$t_1$}%
}}}}
\put(5476,-5461){\makebox(0,0)[lb]{\smash{{\SetFigFont{9}{10.8}{\rmdefault}{\mddefault}{\updefault}{\color[rgb]{0,0,0}$t_2$}%
}}}}
\put(7426,-5461){\makebox(0,0)[lb]{\smash{{\SetFigFont{9}{10.8}{\rmdefault}{\mddefault}{\updefault}{\color[rgb]{0,0,0}$t_3$}%
}}}}
\put(5626,-2761){\makebox(0,0)[lb]{\smash{{\SetFigFont{9}{10.8}{\rmdefault}{\mddefault}{\updefault}{\color[rgb]{0,0,0}$\rho_t=\theta^{\lfloor\log p_t/\log\theta\rfloor}$}%
}}}}
\put(2026,-1561){\makebox(0,0)[lb]{\smash{{\SetFigFont{9}{10.8}{\rmdefault}{\mddefault}{\updefault}{\color[rgb]{0,0,0}$1$}%
}}}}
\put(2026,-3361){\makebox(0,0)[lb]{\smash{{\SetFigFont{9}{10.8}{\rmdefault}{\mddefault}{\updefault}{\color[rgb]{0,0,0}$\theta$}%
}}}}
\put(2026,-4261){\makebox(0,0)[lb]{\smash{{\SetFigFont{9}{10.8}{\rmdefault}{\mddefault}{\updefault}{\color[rgb]{0,0,0}$\theta^2$}%
}}}}
\put(2026,-4711){\makebox(0,0)[lb]{\smash{{\SetFigFont{9}{10.8}{\rmdefault}{\mddefault}{\updefault}{\color[rgb]{0,0,0}$\theta^3$}%
}}}}
\end{picture}%

%% file: fv4.pstex_t
\begin{picture}(0,0)%
\includegraphics{fv4.pstex}%
\end{picture}%
\setlength{\unitlength}{1973sp}%
\begingroup\makeatletter\ifx\SetFigFont\undefined%
\gdef\SetFigFont#1#2#3#4#5{%
  \reset@font\fontsize{#1}{#2pt}%
  \fontfamily{#3}\fontseries{#4}\fontshape{#5}%
  \selectfont}%
\fi\endgroup%
\begin{picture}(8430,4899)(736,-5623)
\put(3076,-961){\makebox(0,0)[lb]{\smash{{\SetFigFont{9}{10.8}{\rmdefault}{\mddefault}{\updefault}{\color[rgb]{0,0,0}$N=4, K=2,\ \text{and}\ p_T^N=\frac{3}{4}(1-\tfrac{2}{4})^2$}%
}}}}
\put(9151,-5311){\makebox(0,0)[lb]{\smash{{\SetFigFont{9}{10.8}{\rmdefault}{\mddefault}{\updefault}{\color[rgb]{0,0,0}$t$}%
}}}}
\put(751,-2761){\makebox(0,0)[lb]{\smash{{\SetFigFont{9}{10.8}{\rmdefault}{\mddefault}{\updefault}{\color[rgb]{0,0,0}$\eta_0$}%
}}}}
\put(8476,-4336){\makebox(0,0)[lb]{\smash{{\SetFigFont{9}{10.8}{\rmdefault}{\mddefault}{\updefault}{\color[rgb]{0,0,0}$X_T^i$}%
}}}}
\put(2476,-2086){\makebox(0,0)[lb]{\smash{{\SetFigFont{9}{10.8}{\rmdefault}{\mddefault}{\updefault}{\color[rgb]{0,0,0}$X_0^i$}%
}}}}
\put(8326,-5461){\makebox(0,0)[lb]{\smash{{\SetFigFont{9}{10.8}{\rmdefault}{\mddefault}{\updefault}{\color[rgb]{0,0,0}$T$}%
}}}}
\put(6676,-5461){\makebox(0,0)[lb]{\smash{{\SetFigFont{9}{10.8}{\rmdefault}{\mddefault}{\updefault}{\color[rgb]{0,0,0}$\tau_2$}%
}}}}
\put(4576,-5461){\makebox(0,0)[lb]{\smash{{\SetFigFont{9}{10.8}{\rmdefault}{\mddefault}{\updefault}{\color[rgb]{0,0,0}$\tau_1$}%
}}}}
\end{picture}%